\documentclass{amsart}
\usepackage[foot]{amsaddr}
\usepackage{amssymb,amsmath,amsfonts,mathptmx,cite,mathrsfs,graphicx,rotating,url}
\usepackage{eucal}
\usepackage{mathdots}
\usepackage{comment}
\usepackage{shuffle}

\usepackage{tikz}
\usetikzlibrary{arrows}
\usetikzlibrary{arrows,shapes}
\usepackage{tikz-cd}

\usepackage[final]{hyperref}

\hypersetup{
  colorlinks = true, 
  urlcolor   = green, 
  linkcolor  = red, 
  citecolor  = blue 
}

\DeclareMathOperator{\alf}{con}

\DeclareMathOperator{\dom}{dom}
\DeclareMathOperator{\var}{var}
\DeclareMathOperator{\pvar}{pvar}
\DeclareMathOperator{\End}{End}

\theoremstyle{plain}
\newtheorem{theorem}{Theorem}[section]
\newtheorem{proposition}[theorem]{Proposition}
\newtheorem{lemma}[theorem]{Lemma}
\newtheorem{corollary}[theorem]{Corollary}

\theoremstyle{definition}

\newtheorem{remark}[theorem]{Remark}

\theoremstyle{remark}
\newtheorem{claim}{Claim}
\newtheorem{step}{Step}

\DeclareSymbolFont{rsfscript}{OMS}{rsfs}{m}{n}
\DeclareSymbolFontAlphabet{\mathrsfs}{rsfscript}

\def\ib{identity basis}

\newcommand{\fb}{finitely based}
\newcommand{\nfb}{non\-finitely based}

\newcommand{\pat}{partial transformation}

\newcommand{\ais}{ai-semi\-ring}

\newcommand{\mA}{\mathcal{A}}
\newcommand{\mB}{\mathcal{B}}
\newcommand{\mC}{\mathcal{C}}
\newcommand{\mF}{\mathcal{F}}

\newcommand{\mP}{\mathcal{P}}
\newcommand{\mQ}{\mathcal{Q}}
\newcommand{\mR}{\mathcal{R}}
\newcommand{\mS}{\mathcal{S}}
\newcommand{\mSt}{\mathcal{S}t}
\newcommand{\mT}{\mathcal{T}}
\newcommand{\mUT}{\mathcal{UT}}
\newcommand{\mZ}{\mathcal{Z}}

\makeatletter
\@namedef{subjclassname@2020}{%
  \textup{2020} Mathematics Subject Classification}
\makeatother

\numberwithin{equation}{section}

\begin{document}

\title{Catalan Semirings}
\author[S. V. Gusev]{Sergey V. Gusev}
\address[S. V. Gusev]{\normalfont Institute of Natural Sciences and Mathematics, Ural Federal University, 620000 Ekaterinburg, Russia}
\email{sergey.gusb@gmail.com}
\address[M.V.Volkov]{}
\email{mishavolkov@gmail.com}

\author[M. V. Volkov]{Mikhail V. Volkov}

\begin{abstract}
We study Catalan semirings, which are additively idempotent semirings whose multiplicative semigroups belong to a natural family of transformation monoids known as Catalan monoids. There are two kinds of Catalan semirings, whose properties differ substantially. We address the Finite Basis Problem for individual Catalan semirings of each kind and their representations by Boolean triangular matrices. Furthermore, we characterize the varieties and pseudovarieties generated by all Catalan semirings of each kind.
\end{abstract}



\date{October 2026}

\thanks{Sergey Gusev was supported by the Ministry of Science and Higher Education of the Russian Federation (project FEUZ-2026-0007)}

\maketitle

\tableofcontents

\section*{Introduction}
\label{sec:intro1}

If the collection of all identities satisfied by an algebra $\mA$ is finitely axiomatized, one says that $A$ \emph{admits a finite identity basis} or is \emph{\fb}. Otherwise, $\mA$ is said to be \emph{\nfb}. The \emph{Finite Basis Problem} for a class of algebras asks for a classification of the algebras in that class according to whether they are \fb{} or \nfb.

A surprising discovery at the time, due to Lyndon~\cite{Lyndon:1954}, was that a finite algebra could be \nfb, a thought-provoking appearance of infinity hiding within a finite object. Lyndon's \nfb{} algebra had seven elements. Later, smaller examples were found, including groupoids (i.e., algebras with a single binary operation) with just three elements exhibited by Murski\v{\i} \cite{Murskii:1965} and Je\v{z}ek \cite{Jezek:1985}. These examples reached the limit in the sense that all 2-element algebras are \fb, another classic result due to Lyndon~\cite{Lyndon:1951}; see also  a modernized proof by Berman \cite{Berman:1980}.

Recently, Jackson, Ren, and Zhao~\cite{JackRenZhao:2022} found a 3-element \nfb{} algebra $\mS_7=\left(\{\mathit{0},a,\mathit{1}\};\,+,\cdot\right)$, where for $x,y\in\{\mathit{0},a,\mathit{1}\}$,
\begin{equation}
\label{eq:s7}
x+y:=\begin{cases}
      x &\text{if } x=y, \\
      \mathit{0} &\text{otherwise};
    \end{cases}
    \qquad
x\cdot y:=\begin{cases}
      x &\text{if } y=\mathit{1}, \\
      y &\text{if } x=\mathit{1}, \\
      \mathit{0} &\text{otherwise}.
    \end{cases}
\end{equation}
$\mS_7$ has the minimum possible size for a \nfb{} algebra, but its operations are transparent and much better behaved than those of the groupoids of Murski\v{\i} and Je\v{z}ek: both operations are commutative and associative, addition is idempotent, and multiplication distributes over addition. $\mS_7$ is thus what is called an {additively idempotent semiring}.

In general, an \emph{additively idempotent semiring} (\ais, for short) is an algebra $(S;\,+,\cdot)$ with two binary operations, addition $+$ and multiplication $\cdot$, such that
\begin{itemize}
\item the additive reduct $(S;\,+)$ is a \emph{semilattice}, that is,
\[
s+(t+u)=(s+t)+u,\ \ s+t=t+s,\ \text{ and }\ s+s=s\ \text{ for all }\ s,t,u\in S;
\]
\item the multiplicative reduct $(S;\,\cdot)$ is a \emph{semigroup}, that is,
\[
s(tu)=(st)u\ \text{ for all }\ s,t,u\in S;
\]
\item multiplication distributes over addition on the left and on the right, that is,
\[
s(t+u)=st+su\ \text{ and }\ (s+t)u=su+tu\ \text{ for all }\ s,t,u\in S.
\]
\end{itemize}

Ai-semirings naturally arise in various areas of mathematics (e.g., idempotent analysis, tropical geometry, and optimization) and computer science; see the contributions in the conference volumes \cite{Guna98,LiMa:2005,LiSe:2009} and the monograph \cite{Bis04} for diverse examples. The study of ai-semirings in their own right has developed into a rich theory, revealing many interesting phenomena, in particular, those related to the Finite Basis Problem.

We have already mentioned the \nfb{} \ais{} $S_7$. It is known to be the only 3-element \nfb{} \ais; see \cite{Zhaoetal:2020}. Thus, the Finite Basis Problem for the class of all \ais{}s with at most three elements has been settled. This solution can be viewed as a step towards the classification of finite \ais{}s as \fb{} or \nfb{}, which has become a very active area of research since the mid-2000s and especially in recent years. We mention \cite{Dolinka07,Dolinka09a,Dolinka09b,Dolinka09c,DolinkaGusVol24,DolinkaGusVol25,GJRZ26,GaoRen26,GRSY26,GRZhao26,GRZhou26,GZhou26,GZhou26a,GPZ05,GuVo23a,GuVo23b,GuVo26,JackRenZhao:2022,JiaoRen26,KP10,MaYue26,Pas05,
RJZL23,RLYC26,RLZC25,RYZ26,RZS16,RZW17,RZ19,SG26,SRG26,SY26,ShaoRen15,Shap24,Shap26,Vol21,WW26,WYW26,WYW26a,WuRenZhao24,WuZhaoRen23,yue26,yr25,yr26,yrg26,yrg26a,yrzs25,ys26,ys26a,Zhaoetal:2020,ZhaoWu25} as examples of works in this direction. The present paper contributes to the area by tackling the Finite Basis Problem for finite \ais{}s whose multiplicative reducts belong to a natural family of transformation monoids known as Catalan monoids. There are two kinds of such \ais{}s, whose properties differ substantially; the identities of the \ais{}s of one of these two kinds have not yet been studied.

The paper is structured as follows. In Section~\ref{sec:catalan}, we define Catalan monoids and show how they can be equipped with an addition that makes them \ais{}s, naturally called Catalan semirings. In Section~\ref{sec:embedding}, we discuss faithful representations of Catalan semirings by upper triangular Boolean matrices; we construct a new such representation, which is used in the proof of one of our main results and also appears to be of independent interest. Section~\ref{sec:identities} provides the necessary preliminaries on \ais{} identities. Section~\ref{sec:fbp} presents our findings related to the Finite Basis Problem for individual Catalan semirings, and in Section~\ref{sec:varieties} we address the same problem for Catalan semirings in bulk. We conclude in Section~\ref{sec:conclusion} with a discussion of several related results and an overview of open questions.

We have made an effort to keep the paper reasonably self-contained. In particular, no specific background in semigroup or semiring theory is assumed. The few necessary concepts from universal algebra are recalled when invoked.

\section{Catalan monoids and Catalan semirings}
\label{sec:catalan}

\subsection{Catalan monoids}
\label{subsec:catalanmonoids}
Let $[n]$ stand for the set $\{1,2,\dots,n\}$ of the first $n$ positive integers equipped with their usual order: $1<2<\dots<n$. A transformation $\alpha\colon[n]\to[n]$ is \emph{order-preserving} if $i\le j$ implies $i\alpha\le j\alpha$ for all $i,j\in[n]$ and \emph{extensive} if $i\le i\alpha$ for every $i\in[n]$. Clearly, if two transformations have either of the properties of being order-preserving or extensive, then so does their product, and the identity transformation has both properties. Hence, the set $C_n$ of all extensive order-preserving transformations of $[n]$ forms a submonoid of the monoid of all transformations of $[n]$. For illustration, here are the five transformations that form the submonoid $C_3$:

\medskip

\begin{center}
\begin{tikzpicture}
[scale=0.9]
\foreach \x in {0.75,2.25,3.5,5,6.25,7.75,9,10.5,11.75,13.25} \foreach \y in {0.5,1.5,2.5} \filldraw (\x,\y) circle (2pt);
	\draw (0.75,0.5) edge[-latex] (2.25,0.5);
	\draw (0.75,1.5) edge[-latex] (2.25,1.5);
    \draw (0.75,2.5) edge[-latex] (2.25,2.5);
	\draw (3.5,0.5) edge[-latex] (5,1.5);
	\draw (3.5,1.5) edge[-latex] (5,1.5);
    \draw (3.5,2.5) edge[-latex] (5,2.5);
	\draw (6.25,0.5) edge[-latex] (7.75,0.5);
	\draw (6.25,1.5) edge[-latex] (7.75,2.5);
    \draw (6.25,2.5) edge[-latex] (7.75,2.5);
	\draw (9,0.5) edge[-latex] (10.5,1.5);
	\draw (9,1.5) edge[-latex] (10.5,2.5);
    \draw (9,2.5) edge[-latex] (10.5,2.5);
	\draw (11.75,0.5) edge[-latex] (13.25,2.5);
	\draw (11.75,1.5) edge[-latex] (13.25,2.5);
    \draw (11.75,2.5) edge[-latex] (13.25,2.5);
\foreach \x in {0.5,2.5,3.25,5.25,6,8,8.75,10.75,11.5,13.5} \node at (\x,0.5) {\tiny 1};
\foreach \x in {0.5,2.5,3.25,5.25,6,8,8.75,10.75,11.5,13.5} \node at (\x,1.5) {\tiny 2};
\foreach \x in {0.5,2.5,3.25,5.25,6,8,8.75,10.75,11.5,13.5} \node at (\x,2.5) {\tiny 3};
\end{tikzpicture}
\end{center}

\medskip

\noindent The submonoid $C_n$ often appears in the literature as the \emph{Catalan} monoid. The name is justified by the size of $C_n$ being the $n$-th Catalan number $\frac1{n+1}\binom{2n}n$; for proofs of this fact, see \cite[Theorem 3.1]{Hi93} or \cite[Theorem 14.2.8(i)]{GaMa09} or Subsection~\ref{subsec:young} below.

A transformation $\alpha\colon[n]\to[n]$ is called \emph{decreasing} (or \emph{parking}) if $i\alpha\le i$ for every $i\in[n]$. The set $C^-_n$ of all decreasing order-preserving transformations of $[n]$ also forms a submonoid of the monoid of all transformations of $[n]$. The submonoid $C^-_n$ differs from the submonoid $C_n$ when $n>1$ (which we will always assume), but the two monoids are isomorphic because $C^-_n$ is nothing more than the monoid of all extensive order-preserving transformations of the set $\{1,2,\dots,n\}$ equipped with the `negative' order $1>2>\dots>n$. For illustration, here are the five transformations that form the submonoid $C^-_3$:

\medskip

\begin{center}
\begin{tikzpicture}
[scale=0.9]
\foreach \x in {0.75,2.25,3.5,5,6.25,7.75,9,10.5,11.75,13.25} \foreach \y in {0.5,1.5,2.5} \filldraw (\x,\y) circle (2pt);
	\draw (0.75,0.5) edge[-latex] (2.25,0.5);
	\draw (0.75,1.5) edge[-latex] (2.25,1.5);
    \draw (0.75,2.5) edge[-latex] (2.25,2.5);
	\draw (3.5,0.5) edge[-latex] (5,0.5);
	\draw (3.5,1.5) edge[-latex] (5,1.5);
    \draw (3.5,2.5) edge[-latex] (5,1.5);
	\draw (6.25,0.5) edge[-latex] (7.75,0.5);
	\draw (6.25,1.5) edge[-latex] (7.75,0.5);
    \draw (6.25,2.5) edge[-latex] (7.75,2.5);
	\draw (9,0.5) edge[-latex] (10.5,0.5);
	\draw (9,1.5) edge[-latex] (10.5,0.5);
    \draw (9,2.5) edge[-latex] (10.5,1.5);
	\draw (11.75,0.5) edge[-latex] (13.25,0.5);
	\draw (11.75,1.5) edge[-latex] (13.25,0.5);
    \draw (11.75,2.5) edge[-latex] (13.25,0.5);
\foreach \x in {0.5,2.5,3.25,5.25,6,8,8.75,10.75,11.5,13.5} \node at (\x,0.5) {\tiny 1};
\foreach \x in {0.5,2.5,3.25,5.25,6,8,8.75,10.75,11.5,13.5} \node at (\x,1.5) {\tiny 2};
\foreach \x in {0.5,2.5,3.25,5.25,6,8,8.75,10.75,11.5,13.5} \node at (\x,2.5) {\tiny 3};
\end{tikzpicture}
\end{center}

\smallskip

In some sources, it is the monoid $C^-_n$ that is called the Catalan monoid. Slightly abusing terminology, we collectively refer to the monoids in both series $\{C_n\}_{n=2,3,\dotsc}$ and $\{C^-_n\}_{n=2,3,\dotsc}$ as Catalan monoids.

Catalan monoids have a strong relationship with certain important classes of recognizable languages and were therefore initially studied within the framework of the Eilenberg correspondence between so-called varieties of recognizable languages and pseudovarieties of finite monoids; see, e.g., \cite[Chapter~IV]{Pin86}. Due to their remarkable combinatorial properties and connections to Coxeter groups, Catalan monoids have subsequently been investigated from various viewpoints (and under various names). We mention \cite{AAK11,DHST11,GaMa11,GaoZhangLuo22,Gr12,GrMa14,Hi93,HiTh09,lebed2025,LuoJinZhang25,MarSt18,SaVo23,So96,Stein26,Ts90,Vo04} as examples of such studies.

\subsection{Catalan semirings} The set $O_n$ of all order-preserving transformations of $[n]$ is a submonoid of the monoid of all transformations of $[n]$; clearly, both $C_n$ and $C^-_n$ are submonoids of the monoid $O_n$. We define, for all $\alpha,\beta\in O_n$ and $i\in[n]$,
\[
i(\alpha+\beta):=\max\{i\alpha,i\beta\}.
\]
It is well known (and easy to verify) that $\alpha+\beta\in O_n$. Equipped with this addition, $O_n$ becomes an \ais{} in which both $C_n$ and $C^-_n$ form subsemirings. Hence, Catalan monoids admit a natural \ais{} structure. We denote the \ais{}s defined on $C_n$ and $C^-_n$ by $\mC_n$ and $\mC^-_n$, respectively, and refer to them as \emph{Catalan} semirings. Notice that while for each $n=2,3,\dotsc$, the monoids $C_n$ and $C^-_n$ are isomorphic, the semirings $\mC_n$ and $\mC^-_n$ are not isomorphic, and moreover, their additive reducts are not isomorphic if $n>2$.

For $\alpha,\beta\in O_n$, let
\begin{equation}\label{eq:semilattice order}
\alpha\leqslant \beta\iff \alpha+\beta=\beta\iff i\alpha\le i\beta\ \text{ for all }\  i\in[n].
\end{equation}
The relation $\leqslant$ is a partial order on the set $O_n$, under which $O_n$ becomes a distributive lattice. The addition on $O_n$ can be recovered from order \eqref{eq:semilattice order}---for all $\alpha,\beta\in O_n$, their sum $\alpha+\beta$ coincides with the least upper bound of $\alpha$ and $\beta$ with respect to $\leqslant$.

For illustration, the next picture shows the Hasse diagram of the lattice $(O_3;\leqslant)$. In it, each transformation $\alpha\in O_3$ is encoded by the triple $1\alpha 2\alpha 3\alpha$; in particular, the identity transformation $\varepsilon$ is encoded by the triple 123.
\begin{center}
\begin{tikzpicture}[x=1.6cm,y=1.0cm, every node/.style={font=\small}]
  \node (111) at (0,0) {$111$};

  \node (112) at (0,1) {$112$};

  \node (113) at (-1,2) {$113$};
  \node (122) at (1,2) {$122$};

  \node (123) at (0,3) {$123$};
  \node (222) at (2,3) {$222$};

  \node (133) at (-1,4) {$133$};
  \node (223) at (1,4) {$223$};

  \node (233) at (0,5) {$233$};

  \node (333) at (0,6) {$333$};

  \draw (111) -- (112);

  \draw (112) -- (113);
  \draw (112) -- (122);

  \draw (113) -- (123);
  \draw (122) -- (123);
  \draw (122) -- (222);

  \draw (123) -- (133);
  \draw (123) -- (223);
  \draw (222) -- (223);

  \draw (133) -- (233);
  \draw (223) -- (233);

  \draw (233) -- (333);

  \node (e) at (-.8,3) {$\varepsilon$};
  \draw (-.7,3) edge[->] (-.25,3);
\end{tikzpicture}
\end{center}

\smallskip

\noindent The five transformations $\alpha$ with $\alpha\geqslant\varepsilon$ form the semiring $\mC_3$, while the five transformations $\beta$ with $\beta\leqslant\varepsilon$ constitute the semiring $\mC^-_3$. The general picture looks the same: the transformations $\alpha\in O_n$ with $\alpha\geqslant\varepsilon$ form the semiring $\mC_n$, while the transformations $\beta\in O_n$ with $\beta\leqslant\varepsilon$ constitute the semiring $\mC^-_n$.

A series of semirings isomorphic to $\mC_n$ appeared in Kl\'ima and Pol\'ak's article~\cite{KP10} in connection with their studies on certain classes of recognizable languages. The semirings $\mC^-_n$ do not seem to have been considered in the literature so far. As the reader will see, their properties neither reduce to nor are dual to those of $\mC_n$, as one might have expected.

\section{Boolean matrix representations for Catalan semirings}
\label{sec:embedding}
This section incorporates some material from the second-named author's preprint~\cite{Vo25a}.

\subsection{Semirings of Boolean matrices} A \emph{Boolean} matrix is a matrix\footnote{In this paper, the word `matrix' always means an $n\times n$ matrix; moreover, we always assume that $n\ge 2$.} with entries $0$ and $1$ only; we assume that $0<1$. The addition and multiplication of such matrices are as usual, except that addition and multiplication of the entries are defined as $x+y:=\max\{x,y\}$ and $x\cdot y:=\min\{x,y\}$. For each $n$, the set of all Boolean $n\times n$ matrices forms an \ais{} under matrix addition and multiplication.

A Boolean matrix $\bigl(a_{ij}\bigr)_{n\times n}$ is called \emph{upper triangular} if $a_{ij}=0$ for all $1\le j<i\le n$. The set of all upper triangular Boolean $n\times n$ matrices is closed under matrix addition and multiplication so it forms a subsemiring of the \ais{} of all Boolean $n\times n$ matrices. We denote by $T_n$ and $\mT_n$ the monoid and, respectively, the \ais{} of all upper triangular Boolean $n\times n$ matrices.

\subsection{\texorpdfstring{A faithful representation $\mC_n\hookrightarrow\mT_n$}{A faithful representation of Cn in Tn}}
\label{subsec:klimapolak}
By a \emph{faithful representation} of a semigroup $(s;\,cdot)$, respectively, \ais{} $(S;\,+,\cdot)$ in the monoid $T_n$, respectively, in the \ais{} $\mT_n$, we mean an injective map $S\hookrightarrow T_n$ which is a semigroup, respectively, semiring homomorphism.

Any transformation $\alpha\colon[n]\to[n]$ can be represented as the Boolean $n\times n$ matrix $B(\alpha)$ whose entry in position $(i,j)$ is 1 if $i\alpha=j$ and 0 otherwise. The map $\alpha\mapsto B(\alpha)$ is an injective homomorphism from the monoid of all transformations of $[n]$ in the monoid of all Boolean $n\times n$ matrices. If $\alpha$ is extensive, then $i\le i\alpha$ whence the matrix $B(\alpha)$ is upper triangular. Therefore, the restriction of $\alpha\mapsto B(\alpha)$ to the monoid $C_n$ is a faithful representation of $C_n$ in the monoid $T_n$ of all upper triangular Boolean $n\times n$ matrices. However, the map $\alpha\mapsto B(\alpha)$ fails to respect addition so that it is not a representation of the semiring $\mC_n$ in the semiring $\mT_n$.

Kl\'ima and Pol\'ak (see \cite[Section~5]{KP10}) found another faithful representation of the monoid $C_n$ in the monoid $T_n$. (The same representation was rediscovered by Johnson and Fenner \cite[Lemma 5.1]{JF19} and, again, by  Luo, Jin, and Zhang \cite[Theorem 5.4]{LuoJinZhang25}.) Namely, let $St_n$ stand for the set of all \emph{staircase triangular} matrices,\footnote{Our notation slightly differs from that of \cite{KP10}, where this set would be denoted by $S_{n-1}$.} that is, upper triangular Boolean matrices $\bigl(a_{ij}\bigr)_{n\times n}$ satisfying: $a_{ii}=1$ for all $i=1,\dots,n$, and if $a_{ij} = 1$ for some $1\le i < j\le n$, then
\[
a_{ii+1}=\cdots=a_{ij}=a_{i+1j}=\cdots=a_{j-1j}=1.
\]
The set $St_n$ is closed under the usual multiplication of Boolean matrices and contains the identity matrix. Thus, $St_n$ forms a monoid, and the map $\alpha\mapsto S(\alpha)$, where the $(i,j)$-th entry of the matrix $S(\alpha)$ equals 1 if and only if $i\le j\le i\alpha$, is an isomorphism between the monoids $C_n$ and $St_n$.

It is easy to see that the set $St_n$ is also closed under the entry-wise addition of Boolean matrices, so it also forms a subsemiring $\mSt_n$ in the semiring $\mT_n$. Although this was not explicitly mentioned in \cite{KP10}, the above isomorphism  $\alpha\mapsto S(\alpha)$ between the monoids $C_n$ and $St_n$ is, in fact, an isomorphism between the \ais{}s $\mC_n$ and $\mSt_n$. Thus, it provides a faithful representation of the semiring $\mC_n$ in the semiring $\mT_n$.

\subsection{\texorpdfstring{A faithful representation $\mC^-_{n+1}\hookrightarrow\mT_n$}{A faithful representation of C{n+1} in Tn}} A Boolean matrix $\bigl(a_{ij}\bigr)_{n\times n}$ is called \emph{lower triangular} if $a_{ij}=0$ for all $1\le i<j\le n$. The set of all lower triangular Boolean $n\times n$ matrices also forms a subsemiring of the \ais{} of all Boolean $n\times n$ matrices; we denote this subsemiring by $\mT^-_n$. If $P$ is the $n\times n$ matrix with 1s on the antidiagonal and 0s elsewhere,
\[
P:=\begin{pmatrix}
0 & 0 & \cdots & 0 & 1\\
0 & 0 & \cdots & 1 & 0\\
\vdots & \vdots & \iddots & \vdots & \vdots\\
0 & 1 & \cdots & 0 & 0\\
1 & 0 & \cdots & 0 & 0\\
\end{pmatrix},
\]
then it is known (and easy to verify) that the map $M\mapsto PMP$ is an isomorphism between the \ais{}s $\mT^-_n$ and $\mT_n$. Therefore, to obtain a faithful representation of the semiring $\mC^-_{n+1}$ in $\mT_n$, it suffices to give a faithful representation in $\mT^-_n$, which is notationally easier.

To each transformation $\alpha\in\mC^-_{n+1}$, we assign a Boolean $n\times n$ matrix $M(\alpha)$ whose \mbox{$i$-th} row consists of $(i+1)\alpha-1$ left-justified 1s; in other words, the $i$-th row is the unary representation of the number $(i+1)\alpha-1$, left-justified. The rule can also be expressed by saying that the $(i,j)$-th entry of the matrix $M(\alpha)$ equals 1 if and only if $(i+1)\alpha-1\ge j$.

For illustration, the next table shows the map $\alpha\mapsto M(\alpha)$ for transformations in $\mC^-_3$.

\medskip

\[
\begin{array}{c|c|c|c|c|c|}
\text{Transformation}\ \alpha\in\mC^-_3&&&&&\\[-.5em]
\raisebox{5pt}{(encoded as \  $1\alpha 2\alpha 3\alpha$)} & \raisebox{10pt}{111} & \raisebox{10pt}{112} & \raisebox{10pt}{113} & \raisebox{10pt}{122} & \raisebox{10pt}{123} \\[-1em]
& \downarrow & \downarrow & \downarrow & \downarrow & \downarrow \\[-1em]
& & & & & \\
\text{Boolean matrix } M(\alpha)&
\begin{pmatrix}0 & 0 \\ 0 & 0\end{pmatrix} &
\begin{pmatrix}0 & 0 \\ 1 & 0\end{pmatrix} &
\begin{pmatrix}0 & 0 \\ 1 & 1\end{pmatrix} &
\begin{pmatrix}1 & 0 \\ 1 & 0\end{pmatrix} &
\begin{pmatrix}1 & 0 \\ 1 & 1\end{pmatrix}
\end{array}
\]

\medskip

\begin{theorem}
\label{thm:construction}
The map $\alpha\mapsto M(\alpha)$ is a faithful representation of the semiring $\mC^-_{n+1}$ in the semiring $\mT^-_n$.
\end{theorem}

\begin{proof}
As $\alpha\in\mC^-_{n+1}$ is a decreasing transformation, $(i+1)\alpha\le i+1$ for each $i=1,\dots,n$. Hence,
$(i+1)\alpha-1\le i$ and one must have $j\le i$ for the condition $(i+1)\alpha-1\ge j$ to hold. Thus, only the $i$ first entries in the $i$-th row of the matrix $M(\alpha)$ may be non-zero, and therefore, $M(\alpha)$ is a lower triangular matrix.

It is clear that the map $\alpha\mapsto M(\alpha)$ is one-to-one and preserves addition. To verify that $M(\alpha\beta)=M(\alpha)M(\beta)$ for all transformations $\alpha,\beta\in\mC^-_{n+1}$, let $M(\alpha):=\bigl(a_{ij}\bigr)_{n\times n}$, $M(\beta):=\bigl(b_{ij}\bigr)_{n\times n}$, $M(\alpha\beta):=\bigl(c_{ij}\bigr)_{n\times n}$, and $M(\alpha)M(\beta):=\bigl(d_{ij}\bigr)_{n\times n}$. For all $i,j\in\{1,\dots,n\}$,
\begin{align}
 \notag d_{ij}=\sum_{k=1}^{n} a_{ik}b_{kj}=1&\iff \exists k\ (a_{ik}=1\ \& \ b_{kj}=1)\\
 \notag &\iff \exists k\  \bigl((i+1)\alpha-1\ge k\ \& \  (k+1)\beta-1\ge j\bigr)\\
 \label{eq:statement}&\iff \exists k\  \bigl((i+1)\alpha\ge k+1\ \& \  (k+1)\beta\ge j+1\bigr).
\end{align}
The statement \eqref{eq:statement} is equivalent to the inequality
\begin{equation}
    (i+1)\alpha\beta\ge j+1.\label{eq:product}
\end{equation}
Indeed, if $(i+1)\alpha\ge k+1$ for some $k$, then $(i+1)\alpha\beta\ge (k+1)\beta$ since $\beta$ is order-preserving. Combining the latter inequality with $(k+1)\beta\ge j+1$ yields \eqref{eq:product}. Conversely, \eqref{eq:product} implies that \eqref{eq:statement} holds with $k:=(i+1)\alpha-1$.

The inequality \eqref{eq:product} is equivalent to $(i+1)\alpha\beta-1\ge j$, and the latter inequality is equivalent to $c_{ij}=1$ by the definition of the map $\alpha\mapsto M(\alpha)$. Hence $d_{ij}=1\iff c_{ij}=1$, and the matrices $M(\alpha)M(\beta)$ and $M(\alpha\beta)$ coincide.
\end{proof}

\begin{corollary}
\label{cor:C_(n+1)toT(n)}
The map $\alpha\mapsto PM(\alpha)P$ is a faithful representation of the semiring $\mC^-_{n+1}$ in the semiring $\mT_n$.
\end{corollary}

\subsection{Combinatorial application}
\label{subsec:young}
The lower triangular Boolean matrices in the image of bijection $\alpha\mapsto M(\alpha)$ are in a natural one-to-one correspondence with the Young diagrams contained in the staircase shape $(n,n-1,\dots,1)$. The next illustration shows these matrices together with the corresponding Young diagrams (drawn in the French convention---rows are left-justified and counted from bottom to top) for $n=2$.

\def\boxsize{0.6cm}

\newcommand{\drawBounding}{%
  \draw[very thick,black] (0,0) rectangle ++(\boxsize,\boxsize);
  \draw[very thick,black] (\boxsize,0) rectangle ++(\boxsize,\boxsize);
  \draw[very thick,black] (0,\boxsize) rectangle ++(\boxsize,\boxsize);
}

\[
\begin{array}{cccccc}
&
\begin{pmatrix}0 & 0 \\ 0 & 0\end{pmatrix} &
\begin{pmatrix}0 & 0 \\ 1 & 0\end{pmatrix} &
\begin{pmatrix}0 & 0 \\ 1 & 1\end{pmatrix} &
\begin{pmatrix}1 & 0 \\ 1 & 0\end{pmatrix} &
\begin{pmatrix}1 & 0 \\ 1 & 1\end{pmatrix} \\[-.6em]
& & & & & \\
& \downarrow & \downarrow & \downarrow & \downarrow & \downarrow \\
& & & & & \\[-.6em]
& 
\begin{tikzpicture}
  \drawBounding
\end{tikzpicture}
& 
\begin{tikzpicture}
   \fill[fill=lightgray] (0,0) rectangle ++(\boxsize,\boxsize);
   \drawBounding
\end{tikzpicture}
& 
\begin{tikzpicture}
   \fill[fill=lightgray] (0,0) rectangle ++(\boxsize,\boxsize);
  \fill[fill=lightgray] (\boxsize,0) rectangle ++(\boxsize,\boxsize);
  \drawBounding
\end{tikzpicture}
& 
\begin{tikzpicture}
   \fill[fill=lightgray] (0,0) rectangle ++(\boxsize,\boxsize);
  \fill[fill=lightgray] (0,\boxsize) rectangle ++(\boxsize,\boxsize);
  \drawBounding
\end{tikzpicture}
& 
\begin{tikzpicture}
  \fill[fill=lightgray] (0,0) rectangle ++(\boxsize,\boxsize);
  \fill[fill=lightgray] (\boxsize,0) rectangle ++(\boxsize,\boxsize);
  \fill[fill=lightgray] (0,\boxsize) rectangle ++(\boxsize,\boxsize);
  \drawBounding
\end{tikzpicture}
\end{array}
\]
\smallskip

\noindent The number of Young diagrams contained in the staircase shape $(n,n-1,\dots,1)$ is known to be the $(n+1)$-st Catalan number; see \cite[Exercise 167]{Stanley}. Therefore, the established correspondence gives a direct bijective proof of the `namesake' result about the size of Catalan monoids. This proof is different from the published proofs of the equality $|C_n|=\frac1{n+1}\binom{2n}n$ given in \cite[Theorem 3.1]{Hi93} and \cite[Theorem 14.2.8(i)]{GaMa09}, neither of which is bijective. An alternative bijective proof can be extracted from \cite[Exercise 78]{Stanley}.

\subsection{Minimality}
\label{subsec:optimality}
The faithful representation of Corollary \ref{cor:C_(n+1)toT(n)} is somewhat unexpected because of its `squeezing' feature: an object of a larger dimension embeds into one of smaller dimension. This raises the question: Can Catalan monoids or semirings be faithfully represented by Boolean triangular matrices of even smaller size? We can show that both Kl\'ima--Pol\'ak's embedding $\mC_n\hookrightarrow\mT_n$ and our embedding $\mC^-_{n+1}\hookrightarrow\mT_n$ already use triangular matrices of the minimum possible size.

\begin{proposition}
\label{prop:optimality}
\emph{(i)} The Catalan monoid $C_{n+1}$ does not embed into the monoid $T_{n-1}$.

\emph{(ii)} Neither of the Catalan semirings $\mC_{n}$ and $\mC^-_{n+1}$ embeds into the semiring $\mT_{n-1}$.
\end{proposition}

We defer the proof of Proposition~\ref{prop:optimality} to Subsection~\ref{subsec:idTn}, since the proof relies on certain identities that hold in $T_{n-1}$ and $\mT_{n-1}$ and are presented in that subsection.

\subsection{Complementarity}
\label{subsec:complementarity}
While the embedding $C^-_{n+1}\hookrightarrow T_n$ defined by $\alpha\mapsto PM(\alpha)P$ appears to be new, it is closely related to Kl\'ima--Pol\'ak's embedding and may be viewed, in a sense, as the complement of the embedding $C_{n+1}\hookrightarrow T_{n+1}$ defined by $\alpha\mapsto S(\alpha)$. Let us explain how this complementarity works.

The monoid $O_{n+1}$ of all order-preserving transformations of the chain $[n+1]$ admits the following automorphism $\alpha\mapsto\overline{\alpha}$ induced by reversing the order on $[n+1]$: for $\alpha\in O_{n+1}$,
the transformation $\overline{\alpha}$ is defined by $i\overline{\alpha}:=n+2-(n+2-i)\alpha \ \text{ for each } \ i=1,\dots,n+1$. The restriction of $\alpha\mapsto\overline{\alpha}$ to each of the Catalan monoids $C_{n+1}$ and $C^-_{n+1}$ is an isomorphism onto the other monoid. For $\alpha\in C_{n+1}$, `negate' the upper triangle of the $(n+1)\times(n+1)$ matrix $S(\alpha)$ by replacing all 1s in this triangle by 0s and all 0s with 1s; the 0s below the main diagonal remain unchanged. Removing the first column and the last row from the resulting matrix yields an $n\times n$ matrix, which coincides with the matrix $PM(\overline{\alpha})P$.

\medskip

\begin{center}
\begin{tikzpicture}[scale=1.3, every node/.style={minimum width=2cm, minimum height=2cm, align=center}]

\node (A) at (90:3)
{$\begin{pmatrix}1&0&0&0\\
0&1&0&0\\
0&0&1&1\\
0&0&0&1
\end{pmatrix}$};
\node (B) at (30:3)
{$\begin{pmatrix}0&1&1&1\\
0&0&1&1\\
0&0&0&0\\
0&0&0&0
\end{pmatrix}$};
\node (C) at (330:3)
{$\begin{pmatrix}1&1&1\\
0&1&1\\
0&0&0
\end{pmatrix}$};
\node (D) at (270:3)
{$\begin{pmatrix}0&0&0\\
1&1&0\\
1&1&1
\end{pmatrix}$};
\node (E) at (210:3)
{\begin{tikzpicture}[scale=0.5]
\useasboundingbox (1,2) rectangle (2.2,2.6);
\foreach \x in {0.5,2.5} \node at (\x,0.5) {\tiny 1};
\foreach \x in {0.5,2.5} \node at (\x,1.5) {\tiny 2};
\foreach \x in {0.5,2.5} \node at (\x,2.5) {\tiny 3};
\foreach \x in {0.5,2.5} \node at (\x,3.5) {\tiny 4};
    \draw (0.75,0.5) edge[->] (2.25,0.5);
	\draw (0.75,1.5) edge[->] (2.25,0.5);
    \draw (0.75,2.5) edge[->] (2.25,2.5);
    \draw (0.75,3.5) edge[->] (2.25,3.5);
    \node at (1.5,-0.15)  {$\overline{\alpha}$};
\end{tikzpicture}};
\node (F) at (150:3)
{\begin{tikzpicture}[scale=0.5]
\useasboundingbox (1,0) rectangle (2.2,2.6);
\foreach \x in {0.5,2.5} \node at (\x,0.5) {\tiny 1};
\foreach \x in {0.5,2.5} \node at (\x,1.5) {\tiny 2};
\foreach \x in {0.5,2.5} \node at (\x,2.5) {\tiny 3};
\foreach \x in {0.5,2.5} \node at (\x,3.5) {\tiny 4};
    \draw (0.75,0.5) edge[->] (2.25,0.5);
	\draw (0.75,1.5) edge[->] (2.25,1.5);
    \draw (0.75,2.5) edge[->] (2.25,3.5);
    \draw (0.75,3.5) edge[->] (2.25,3.5);
\node at (1.5,4.15)   {$\alpha$};
\end{tikzpicture}};

\draw[->, thick] (A) -- node[sloped, above, yshift=-4ex] {`negation'} (B);
\draw[->, thick] (B) -- node[align=center, right] {Removing\\  the first column \\ and the last row}(C);
\draw[<-, thick] (C) --  node[sloped, above, yshift=-4ex] {$PM(\overline{\alpha})P$} (D);
\draw[<-, thick] (D) --  node[sloped, above, yshift=-4ex] {$M(\overline{\alpha})$} (E);
\draw[<->, thick] (E) -- node[left] {automorphism} (F);
\draw[->, thick] (F) --  node[sloped, above, yshift=-4ex] {$S(\alpha)$} (A);

\end{tikzpicture}
\end{center}

The above illustration shows all steps of the process for $\alpha=1244\in C_4$. We restrict ourselves to this `proof by example' since a formal verification would merely require unfolding the definitions of all the maps involved.

Any Boolean matrix of a fixed size is fully determined by the set of its entries equal to 1, so it can be identified with this set. Under this identification,  the complementarity described above means that $S(\alpha)$ and $PM(\overline{\alpha})P$ can be viewed as $B(\alpha)$ extended to the left and to the right, respectively. To obtain $S(\alpha)$, one fills the region bounded by the main diagonal and $B(\alpha)$ (including the boundary) with 1s, whereas for $PM(\overline{\alpha})P$, one fills the region to the right of $B(\alpha)$ with 1s and then removes the first column and the last row from the resulting matrix. For illustration, with $\alpha=335666\in C_6$, filling the dark-blue cells of the following schematic $6\times 6$ matrix with 1s gives $B(\alpha)$; adding 1s in the light-blue cells yields $S(\alpha)$; to obtain the $5\times 5$ matrix $PM(\overline{\alpha})P$, one places 1s only in the red cells and removes the first column and the last row.

\medskip

\begin{center}
\begin{tikzpicture}[scale=0.8]
\foreach \x in {-6.5,-4.5} \foreach \y in {0.5,1.5,2.5,3.5,4.5,5.5} \filldraw (\x,\y) circle (2pt);
\foreach \x in {-6.75,-4.25} \node at (\x,0.5) {\tiny 1};
\foreach \x in {-6.75,-4.25} \node at (\x,1.5) {\tiny 2};
\foreach \x in {-6.75,-4.25} \node at (\x,2.5) {\tiny 3};
\foreach \x in {-6.75,-4.25} \node at (\x,3.5) {\tiny 4};
\foreach \x in {-6.75,-4.25} \node at (\x,4.5) {\tiny 5};
\foreach \x in {-6.75,-4.25} \node at (\x,5.5) {\tiny 6};
    \draw (-6.5,0.5) edge[-latex] (-4.5,2.5);
	\draw (-6.5,1.5) edge[-latex] (-4.5,2.5);
    \draw (-6.5,2.5) edge[-latex] (-4.5,4.5);
    \draw (-6.5,3.5) edge[-latex] (-4.5,5.5);
    \draw (-6.5,4.5) edge[-latex] (-4.5,5.5);
    \draw (-6.5,5.5) edge[-latex] (-4.5,5.5);

\node at (-2.5,3) {$\mathbf{\longmapsto}$};

  \foreach \i/\j in {1/1,1/2,1/3,2/2,2/3,3/3,3/4,3/5,4/4,4/5,4/6,5/5,5/6} {
    \fill[blue!20] (\j-1, 6-\i) rectangle (\j, 7-\i);
    }
  \foreach \i/\j in {1/4,1/5,1/6,2/4,2/5,2/6,3/6} {
    \fill[red!60] (\j-1, 6-\i) rectangle (\j, 7-\i);
    }
  \foreach \i/\j in {1/3,2/3,3/5,4/6,5/6,6/6} {
    \fill[blue!60] (\j-1, 6-\i) rectangle (\j, 7-\i);
    }
  \foreach \i/\j in {2/1,3/1,3/2,4/1,4/2,4/3,5/1,5/2,5/3,5/4,6/1,6/2,6/3,6/4,6/5} {
   \node at (\j-0.5, 6.5-\i) {0};
  }
  \draw[thick] (-0.5,0) .. controls (-1.0,1.5) and (-1.0,4.5) .. (-0.5,6);
  \draw[thick] (6.5,0) .. controls (7.0,1.5) and (7.0,4.5) .. (6.5,6);
\end{tikzpicture}
\end{center}

\medskip

We remark that the described connection between the monoid embeddings $C^-_{n+1}\hookrightarrow T_n$ and $C_{n+1}\hookrightarrow T_{n+1}$ does not extend to the semiring setting. This is because the monoid automorphism $\alpha\mapsto\overline{\alpha}$ reverses order \eqref{eq:semilattice order}, and thus does not preserve addition.

\section{Semiring identities}
\label{sec:identities}

\subsection{Words} A \emph{word} is a finite sequence of symbols, called \emph{variables}. Throughout the paper, $x,y,z$ (with or without subscripts) stand for variables. We occasionally employ the \emph{empty word}, that is, the empty sequence. Whenever words under consideration are allowed to be empty, we state this explicitly, so that, by default, `word' means `nonempty word'. We adopt the convention that, for any variable $x$, the expression $x^m$ denotes the empty word if $m=0$ and $\underbrace{x\cdots x}_{\text{$m$ times}}$ if $m$ is a positive integer.

We denote words by lowercase boldface letters. If $\mathbf{w}=x_1\cdots x_k$, where the variables $x_1,\dots,x_k$ may repeat, then the set $\{x_1,\dots,x_k\}$ is denoted by $\alf(\mathbf{w})$ and called the \emph{content} of $\mathbf{w}$, and the number $k$ is denoted by $|\mathbf{w}|$ and called the \emph{length} of $\mathbf{w}$. If $\mathbf{w}$ is empty, then we define $\alf(\mathbf{w}):=\varnothing$ and $|\mathbf{w}|:=0$.

Words are multiplied by concatenation: the product of two words $\mathbf{u}=x_1\cdots x_k$ and $\mathbf{v}=y_1\cdots y_\ell$ is defined as $\mathbf{uv}:=x_1\cdots x_ky_1\cdots y_\ell$. Clearly, concatenation is associative, so we need no parentheses when writing down products of several words. If $\mathbf{w}$ is any word and $\mathbf{v}$ is empty, then $\mathbf{wv}:=\mathbf{w}$ and $\mathbf{vw}:=\mathbf{w}$. If $\mathbf{w}=\mathbf{pfs}$ for words $\mathbf{p}$, $\mathbf{f}$ and $\mathbf{s}$ (some of which may be empty), we say that $\mathbf{p}$ is a \emph{prefix}, $\mathbf{f}$ is a \emph{factor}, and $\mathbf{s}$ is a \emph{suffix} of the word $\mathbf{w}$.

Let $\mathbf{u}=x_1\cdots x_k$ and $\mathbf{v}$ be words. We say that $\mathbf{u}$ is a \emph{subword in $\mathbf{v}$} if there exist words $\mathbf{v}_1, \mathbf{v}_2, \dots,  \mathbf{v}_k, \mathbf{v}_{k+1}$ (some of which may be empty) such that
\begin{equation}
\label{eq:occurrence}
\mathbf{v} = \mathbf{v}_1x_1\mathbf{v}_2\cdots \mathbf{v}_kx_k\mathbf{v}_{k+1}.
\end{equation}
Thus, $\mathbf{u}$ as a sequence of variables is a subsequence in $\mathbf{v}$. The relation of being a subword is a partial order referred to as the \emph{subword order}.

We refer to any decomposition of the form \eqref{eq:occurrence} as an \emph{occurrence} of $\mathbf{u}$ as a subword in $\mathbf{v}$. We say that $\mathbf{u}$ occurs in a word $\mathbf{v}$ \emph{with gaps} $G_1,G_2,\dots,G_{k+1}$ if there is an occurrence \eqref{eq:occurrence} of $\mathbf{u}$ as a subword in $\mathbf{v}$ such that $\alf(\mathbf{v}_\ell)=G_\ell$ for each $\ell=1,2,\dots,k+1$. For illustration, the two occurrences of the word $xy$ as a subword in the word $x^2yx$ are:
\[
x^2yx=\begin{cases}
\underline{x}\cdot x\cdot\underline{y}\cdot x,&\text{with gaps $\varnothing,\{x\},\{x\}$;}\\
x\cdot\underline{x}\underline{y}\cdot x,&\text{with gaps $\{x\},\varnothing,\{x\}$.}
\end{cases}
\]

\subsection{Polynomials and identities} A (\emph{semiring}) \emph{polynomial} is a finite (possibly empty) set of semigroup words. We denote polynomials by uppercase boldface letters and write them as formal sums of their elements. Thus, $\mathbf{W}=\mathbf{w}_1+\dots+\mathbf{w}_p$ means nothing but $\mathbf{W}=\{\mathbf{w}_1,\dots,\mathbf{w}_p\}$; in particular, the order of the summands does not matter. We allow polynomials that contain the empty word, which plays the role of the constant term; a~nonempty polynomial without the empty word is called \emph{proper}.

The above use of the term `polynomial' is quite common in the theory of \ais{}s. To avoid possible confusion, we emphasize that, unlike the `usual' polynomials, semiring polynomials are formed over noncommuting variables, and their monomials (that is, words) carry no coefficients.

The \emph{sum} of two polynomials is their union. The \emph{product} of two nonempty polynomials $\mathbf U=\sum_{i}\mathbf u_i$ and $\mathbf V=\sum_{j}\mathbf v_j$ is the polynomial $\mathbf{UV}:=\sum_{i,j}\mathbf u_i\mathbf v_j$.

For a proper polynomial $\mathbf{W}=\mathbf{w}_1+\dots+\mathbf{w}_p$, we let $\alf(\mathbf{W}):=\cup_{i=1}^p\alf(\mathbf{w}_i)$. We say that a word $\mathbf{u}$ of length $k$ \emph{occurs as a subword in $\mathbf{W}$ with gaps} $G_1,G_2,\dots,G_{k+1}$ if $\mathbf{u}$ occurs as a subword with gaps $G_1,G_2,\dots,G_{k+1}$ in one of the words $\mathbf{w}_i$, $i=1,\dots,p$.

If $\mathbf{W}=\mathbf{w}_1+\dots+\mathbf{w}_p$ is a proper polynomial and $(S;\,+,\cdot)$ is an \ais, then any map $\varphi\colon\alf(\mathbf{W})\to S$ is called a \emph{substitution}. The  \emph{value} $\mathbf{W}\varphi$ of $\mathbf{W}$ under $\varphi$ is defined as $\mathbf{W}\varphi=\sum_{i=1}^p\mathbf{w}_i\varphi$, where, for each $i=1,\dots,p$, if $\mathbf{w}_i=x_1\cdots x_k$, then $\mathbf{w}_i\varphi:=x_1\varphi\cdots x_k\varphi$. As $(S;\,+)$ is a semilattice, the value $\mathbf{W}\varphi$ is well-defined: the element $\sum_{i=1}^p\mathbf{w}_i\varphi\in S$ does not depend on how $\mathbf{W}$ is represented as a sum of words.

A (\emph{semiring}) \emph{identity} is a formal expression of the form $\mathbf{W}\approx\mathbf{W}'$ where $\mathbf{W}$ and $\mathbf{W}'$ are proper polynomials. An \ais{} $(S;\,+,\cdot)$ \emph{satisfies} $\mathbf{W}\approx\mathbf{W}'$ (or $\mathbf{W}\approx\mathbf{W}'$ \emph{holds} in $(S;\,+,\cdot)$) if $\mathbf{W}\varphi=\mathbf{W}'\varphi$ for every substitution $\varphi\colon\alf(\mathbf{W})\cup\alf(\mathbf{W}')\to S$.

Given any semilattice $(S;\,+)$, the binary relation $\leqslant$ defined for $s,t\in S$ by
\[
s\leqslant t\ \text{ if and only if }\ s+t=t
\]
is known (and easy to see) to be a partial order on $S$ compatible with addition. We refer to $\leqslant$ as the +-\emph{order}. If $(S;\,+,\cdot)$ is an \ais, then, by the distributive laws, its +-order is also compatible with multiplication.

A (\emph{semiring}) \emph{inequality} is a formal expression of the form $\mathbf{W}\preccurlyeq\mathbf{W}'$, where $\mathbf{W}$ and $\mathbf{W}'$ are proper polynomials referred to as the \emph{lower} and \emph{upper} sides of the inequality. An \ais{} $(S;\,+,\cdot)$ \emph{satisfies} $\mathbf{W}\preccurlyeq\mathbf{W}'$ (or $\mathbf{W}\preccurlyeq\mathbf{W}'$ \emph{holds} in $(S;\,+,\cdot)$) if for every substitution $\varphi\colon\alf(\mathbf{WW}')\to S$, one has $\mathbf{W}\varphi\leqslant\mathbf{W}'\varphi$. By the definition of the +-order, this inequality is equivalent in $(S;\,+,\cdot)$ to the equality $\mathbf{W}\varphi=\mathbf{W}\varphi+\mathbf{W}'\varphi$. Therefore, the inequality $\mathbf{W}\preccurlyeq\mathbf{W}'$ is equivalent to the identity $\mathbf{W}\approx\mathbf{W}+\mathbf{W}'$  in the sense that they hold in the same \ais{}s. In view of this observation, we take the liberty of applying the term `identity' also to expressions of the form $\mathbf{W}\preccurlyeq\mathbf{W}'$.

\subsection{Inference and identity bases}
\label{subsec:basis}
Given any collection $\Sigma$ of semiring identities, we say that an identity $\tau$ is a \emph{consequence of} $\Sigma$, or that $\Sigma$ \emph{implies} $\tau$, or that $\tau$ \emph{follows from} $\Sigma$, if every \ais{} satisfying all identities in $\Sigma$ also satisfies $\tau$. Birkhoff's completeness theorem for equational logic (see \cite[Theorem~14.17]{BuSa81}) shows that this notion (which we have defined semantically) can also be defined syntactically: an identity follows from $\Sigma$ if it can be inferred from the identities of $\Sigma$ by repeated applications of the following rules: substituting a proper polynomial for all occurrences of a variable in an identity, adding or multiplying two identities side-by-side, and using the symmetry and transitivity of equality. Since we deal with \ais{}s, equality is always understood up to commutativity and idempotency of addition; e.g., we have $\mathbf{U}+\mathbf{V}=\mathbf{V}+\mathbf{U}$ and $\mathbf{W}=\mathbf{W}+\mathbf{W}$ for all $\mathbf U,\mathbf V,\mathbf W$.

For an \ais{} $\mS$, we denote by $\mathsf{Id}(\mS)$ the set of all semiring identities satisfied by $\mS$. An \emph{identity basis} for $\mS$ is any set $\Sigma\subseteq\mathsf{Id}(\mS)$ such that every identity of $\mathsf{Id}(\mS)$ is a consequence of $\Sigma$. An \ais{} $\mS$ is said to be \emph{finitely based} if it possesses a finite identity basis; otherwise $\mS$ is called \emph{nonfinitely based}.

\subsection{Reductions}
\label{subsec:reductions}
We say that two identity systems are \emph{equivalent} if every identity of one system is a consequence of the other. Here we present a series of reductions that transform any system of semiring identities into an equivalent system of a specific form. We start with the obvious observation that an identity $\mathbf{W}\approx\mathbf{W}'$ holds in an \ais{} if and only if the \ais{} satisfies both the inequality $\mathbf{W}\preccurlyeq\mathbf{W}'$ and the opposite inequality $\mathbf{W}'\preccurlyeq\mathbf{W}$. Thus, up to equivalence, we may restrict to systems all of whose identities are in fact inequalities. Next, the identity $\mathbf{W}\preccurlyeq\mathbf{W}'$, where $\mathbf{W}=\mathbf{w}_1+\dots+\mathbf{w}_p$, is equivalent to the system of identities $\mathbf{w}_i\preccurlyeq\mathbf{W}'$, $i=1,\dots,p$. Indeed, since obviously $\mathbf{w}_i\preccurlyeq \mathbf{w}_1+\dots+\mathbf{w}_p$, the identity $\mathbf{W}\preccurlyeq\mathbf{W}'$ implies that  $\mathbf{w}_i\preccurlyeq\mathbf{W}'$ for each $i=1,\dots,p$. Conversely, summing all the identities $\mathbf{w}_i\preccurlyeq\mathbf{W}'$ over $i=1,\dots,p$ yields
\[
\mathbf{W}=\mathbf{w}_1+\dots+\mathbf{w}_p\preccurlyeq\underbrace{\mathbf{W}'+\dots+\mathbf{W}'}_{\text{$p$ times}}\approx \mathbf{W}'
\]
by idempotency of addition.

While the above reductions are well known, the one stated in the next lemma does not seem to have appeared in the literature so far.
\begin{lemma}
\label{lem:linearization}
Every identity $\mathbf{u}\preccurlyeq\mathbf{W}$ is equivalent to an identity of the form $x_1\cdots x_k\preccurlyeq\mathbf{W}'$, where $k=|\mathbf{u}|$ and $x_1,\dots,x_k$ are distinct variables.
\end{lemma}

\begin{proof}
Let $\alf(\mathbf{u})=\{y_1,\dots,y_m\}$. Introduce distinct variables $x_1,\dots,x_k$, where $k=|\mathbf{u}|$, and assign them to the occurrences of the variables in $\mathbf{u}$ in the order in which those occurrences appear. This partitions the set $\{x_1,\dots,x_k\}$ into $m$ subsets $Y_1,\dots,Y_m$. For each $i=1,\dots,m$, replace each occurrence of $y_i$ in the identity $\mathbf{u}\preccurlyeq\mathbf{W}$ by the sum $\sum_{x_j\in Y_i}x_j$, leaving other variables that may appear in $\mathbf{W}$ unchanged. This yields the following consequence of $\mathbf{u}\preccurlyeq\mathbf{W}$:
\begin{equation}\label{eq}
\mathbf{u}(\sum_{x_j\in Y_1}x_j,\dots,\sum_{x_j\in Y_m}x_j)\preccurlyeq \mathbf{W}(\sum_{x_j\in Y_1}x_j,\dots,\sum_{x_j\in Y_m}x_j,\dots).
\end{equation}
Using distributivity, expand the lower side of \eqref{eq} to a sum of words. Then one of the summands is the word $x_1\cdots x_k$. Hence, if $\mathbf{W}'$ is the polynomial obtained by expanding the upper side of \eqref{eq}, then the identity $x_1\cdots x_k\preccurlyeq\mathbf{W}'$ is also a consequence of $\mathbf{u}\preccurlyeq\mathbf{W}$.

Conversely, replacing each variable $x_j\in Y_i$ by the variable $y_i$ transforms the identity $x_1\cdots x_k\preccurlyeq\mathbf{W}'$ back into $\mathbf{u}\preccurlyeq\mathbf{W}$. Indeed, this substitution takes the word $x_1\cdots x_k$ to $\mathbf{u}$, while each word in $\mathbf{W}'$ is mapped to a copy of a summand in $\mathbf{W}$. Since addition is idempotent, this gives precisely $\mathbf{W}$. Thus, the identities $\mathbf{u}\preccurlyeq\mathbf{W}$ and $x_1\cdots x_k\preccurlyeq\mathbf{W}'$ are equivalent.
\end{proof}

We call a word \emph{linear} if all its variables are distinct. Accordingly, we refer to identities of the form $x_1\cdots x_k\preccurlyeq\mathbf{V}$, where $x_1\cdots x_k$ is linear, as \emph{identities with a linear lower side}. We can summarize the observations made in this subsection as follows: every system of \ais{} identities is equivalent to a system of identities with a linear lower side.

\section{Identities of Catalan semirings}
\label{sec:fbp}

\subsection{\texorpdfstring{Identities of $\mC_n$}{Identities of Cn}} Both the combinatorial characterization of the identities holding in the semiring $\mC_n$ and an explicit finite basis for these identities can be extracted from the results of Kl\'ima and Pol\'ak \cite{KP10}. Kl\'ima and Pol\'ak studied the identities of the semiring of all staircase triangular Boolean $n\times n$ matrices ($\mSt_n$ in our notation), but, as mentioned in Subsection~\ref{subsec:klimapolak}, the semirings $\mC_n$ and $\mSt_n$ are isomorphic. Therefore, their results apply to the identities of $\mC_n$ as well.

\begin{proposition}[\!{\cite[Propositions 2(iii), 11(iii), and 12]{KP10}}]
\label{prop:idCn}
An identity $\mathbf U\approx\mathbf V$ holds in the semiring $\mC_{n}$ if and only if the polynomials $\mathbf U$ and $\mathbf V$ have the same set of subwords of length less than $n$.
\end{proposition}

\begin{theorem}
\label{thm:basisCn}
The identities
\begin{gather}
\label{eq:xy>x} x\preccurlyeq xy, \\
\label{eq:yx>x} x\preccurlyeq yx,\\
\kappa_n:\, x_1\cdots x_{n}\approx\sum_{i=1}^{n} x_1\cdots x_{i-1}x_{i+1}\cdots x_{n}\notag
\end{gather}
constitute an identity basis for the semiring $\mC_{n}$.
\end{theorem}

\begin{remark}
\label{rem:no1}
The identity basis for $\mC_{n}$, which one can obtain by combining Theorem 3 of \cite{KP10} with Propositions 11(iii) and 12 of that article, consists of the identity $\kappa_n$ only. This is because in \cite{KP10}, Kl\'ima and Pol\'ak allow the use of the \emph{unit polynomial}\/ $\mathbf{1}$, which consists of the empty word. Substituting $y$ for $x_1$ and $\mathbf{1}$ for the other variables in $\kappa_n$ gives the identity $y\approx\mathbf{1}+y$, which is equivalent to $\mathbf{1}\preccurlyeq y$. Multiplying the latter identity by $x$ on the left gives \eqref{eq:xy>x}, while multiplying it by $x$ on the right yields \eqref{eq:yx>x}. Thus, in the setting of~\cite{KP10}, identities \eqref{eq:xy>x} and \eqref{eq:yx>x} are consequences of $\kappa_n$ and so there is no need to include them in the basis.

In contrast, in our setting, identities \eqref{eq:xy>x} and \eqref{eq:yx>x} cannot be deduced from $\kappa_n$ if $n>2$. Indeed, consider the \ais{} $\mZ_2=\left(\{\mathit{0},a\};\,+,\cdot\right)$, where, for $x,y\in\{\mathit{0},a\}$,
\begin{equation*}
x+y:=\begin{cases}
      \mathit{0} &\text{if } x=y=\mathit{0}, \\
      a &\text{otherwise};
    \end{cases}
    \qquad
x\cdot y:=\mathit{0}.
\end{equation*}
It satisfies $\kappa_n$ for every $n>2$, since both sides of this identity evaluate to $\mathit{0}$ in $\mZ_2$. However, substituting $a$ for $x$ and $\mathit{0}$ for $y$ shows that \eqref{eq:xy>x} and \eqref{eq:yx>x} fail in $\mZ_2$ since $\mathit{0}<a$ in $\mZ_2$.
\end{remark}

For completeness, we include a full proof of Theorem~\ref{thm:basisCn}, even though a substantial part of this proof repeats the argument in \cite[Theorem 3]{KP10}.

\begin{proof}[Proof of Theorem~\ref{thm:basisCn}]
Proposition~\ref{prop:idCn} readily implies that the identities $\kappa_n$, $x+xy\approx xy$, and $x+yx\approx yx$ hold in the semiring $\mC_n$. The latter two identities are equivalent to \eqref{eq:xy>x} and \eqref{eq:yx>x}, respectively.

Now we show that every identity that holds in $\mC_n$ follows from those in the statement of the theorem. First, we specify the role of identities \eqref{eq:xy>x} and \eqref{eq:yx>x}.
\begin{lemma}
\label{lem:subword}
For any words $\mathbf{u}$, $\mathbf{v}$ such that $\mathbf{u}$ is a subword of $\mathbf{v}$, the identity $\mathbf{v}\succcurlyeq\mathbf{u}$ is a consequence of \eqref{eq:xy>x} and \eqref{eq:yx>x}.
\end{lemma}

\begin{proof}
Let $\mathbf{u}=x_1\cdots x_k$. By the definition of a subword, there are words $\mathbf{v}_1, \mathbf{v}_2, \dots,  \mathbf{v}_k, \mathbf{v}_{k+1}$ (some of which may be empty) such that
$\mathbf{v} = \mathbf{v}_1x_1\mathbf{v}_2\cdots \mathbf{v}_kx_k\mathbf{v}_{k+1}$. We have
\begin{align*}
\mathbf{v} &=\mathbf{v}_1(x_1\mathbf{v}_2\cdots \mathbf{v}_kx_k\mathbf{v}_{k+1})\succcurlyeq x_1\mathbf{v}_2\cdots \mathbf{v}_kx_k\mathbf{v}_{k+1} &&\text{since $yx\succcurlyeq x$}\\
                     &=x_1(\mathbf{v}_2(x_2\cdots \mathbf{v}_kx_k\mathbf{v}_{k+1}))\succcurlyeq x_1(x_2\cdots \mathbf{v}_kx_k\mathbf{v}_{k+1}) &&\text{since $yx\succcurlyeq x$}\\
                     &\makebox[6.4cm]{\dotfill} &&\\
                     &=(x_1\cdots x_k)\mathbf{v}_{k+1}\succcurlyeq x_1\cdots x_k &&\text{since $xy\succcurlyeq x$}\\
                     &=\mathbf{u},&&
\end{align*}
as required.
\end{proof}

For a word $\mathbf{w}$, denote the polynomial of all subwords of $\mathbf{w}$ (including $\mathbf{w}$ itself) by $\mathbf{Sub(w)}$ and the polynomial of all subwords of $\mathbf{w}$ of length at most $k$ by $\mathbf{Sub}_k\mathbf{(w)}$. If $\mathbf{W}=\mathbf{w}_1+\dots+\mathbf{w}_p$ is a proper polynomial, let $\mathbf{Sub(W)}:=\mathbf{Sub(w})_1+\dots+\mathbf{Sub(w)}_p$ and $\mathbf{Sub}_k\mathbf{(W)}:=\mathbf{Sub}_k\mathbf{(w})_1+\dots+\mathbf{Sub}_k\mathbf{(w)}_p$. Clearly, the inequality $\mathbf{w}\preccurlyeq\mathbf{Sub(w)}$ holds in any \ais, and so does the inequality $\mathbf{W}\preccurlyeq\mathbf{Sub(W)}$. By Lemma~\ref{lem:subword}, inequalities \eqref{eq:xy>x} and \eqref{eq:yx>x} imply the opposite inequality $\mathbf{W}\succcurlyeq\mathbf{Sub(W)}$, which gives the identity $\mathbf{W}\approx\mathbf{Sub(W)}$. Further, modulo the identity $\kappa_n$, every word of length at least $n$ is equal to a sum of its subwords of length less than $n$. We conclude that in the presence of \eqref{eq:xy>x}, \eqref{eq:yx>x}, and $\kappa_n$, the identity $\mathbf{W}\approx\mathbf{Sub}_{n-1}\mathbf{(W)}$ holds for every proper polynomial $\mathbf{W}$.

By Proposition~\ref{prop:idCn}, we have $\mathbf{Sub}_{n-1}\mathbf{(U)}=\mathbf{Sub}_{n-1}\mathbf{(V)}$ for an arbitrary identity $\mathbf U\approx\mathbf V$ that holds in the semiring $\mC_{n}$. Therefore,
\[
\mathbf U\approx\mathbf{Sub}_{n-1}\mathbf{(U)}=\mathbf{Sub}_{n-1}\mathbf{(V)}\approx\mathbf V
\]
is a consequence of \eqref{eq:xy>x}, \eqref{eq:yx>x}, and $\kappa_n$.
\end{proof}

\subsection{\texorpdfstring{Identities of $\mT_{n-1}$}{Identities of Tn}}
\label{subsec:idTn}
It follows from Corollary~\ref{cor:C_(n+1)toT(n)} that the semiring $\mC^-_n$ embeds in the semiring $\mT_{n-1}$ of all upper triangular Boolean $(n-1)\times(n-1)$ matrices. Therefore, $\mC^-_n$ satisfies every identity holding in $\mT_{n-1}$. We will use the combinatorial characterization of such identities given in \cite{Vo25}, which we reproduce now.

\begin{proposition}[\!{\cite[Theorem 2.1]{Vo25}}]
\label{prop:idTn}
An identity $\mathbf U\preccurlyeq\mathbf V$ holds in the semiring $\mT_{n-1}$ if and only if every word of length $m<n-1$ that occurs in\/ $\mathbf U$ with gaps $G_1,G_2,\dots,G_{m+1}$ occurs in $\mathbf V$ with gaps $G'_1,G'_2,\dots,G'_{m+1}$ such that $G'_\ell\subseteq G_\ell$ for~each $\ell=1,2,\dots,m+1$.

\end{proposition}

With Proposition~\ref{prop:idTn} in hand, we can prove the minimality result from Subsection~\ref{subsec:optimality}.

\begin{proof}[Proof of Proposition~\ref{prop:optimality}]
Proposition~\ref{prop:idTn} implies that the following identities hold in the semiring $\mT_{n-1}$:
\begin{gather}
\label{eq:semigroup identity in Tn} x^{n}\approx x^{n-1},\\
\label{eq:semiring identity in Tn} x^{n-2}y^{n-1}+x^{n-1}y^{n-2}\approx  x^{n-2}y^{n-2}.
\end{gather}
We prove only \eqref{eq:semiring identity in Tn}; the proof of \eqref{eq:semigroup identity in Tn} is similar but simpler.

Take any word $\mathbf{u}=z_1\cdots z_m$ with $m<n-1$. If $\mathbf{u}$ occurs as a subword in the polynomial $x^{n-2}y^{n-1}+x^{n-1}y^{n-2}$, then, by definition, $\mathbf{u}$ is a subword in one of the words $x^{n-2}y^{n-1}$ or $x^{n-1}y^{n-2}$. Because of symmetry, it suffices to consider one case; so, assume that $\mathbf{u}$ is a subword of $x^{n-1}y^{n-2}$. We have $\mathbf{u}=x^{p}y^{m-p}$ for some $0\le p\le m$; therefore, any occurrence of $\mathbf{u}$ in $x^{n-1}y^{n-2}$ can be written as
\begin{equation}\label{eq:decompositionxxy}
x^{n-1}y^{n-2}=\overbrace{\mathbf{w}_1z_1\mathbf{w}_2\cdots\mathbf{w}_pz_p\mathbf{w}'_{p+1}}^{x^{n-1}}\cdot\overbrace{\mathbf{w}''_{p+1}z_{p+1}\mathbf{w}_{p+2}\cdots\mathbf{w}_mz_m\mathbf{w}_{m+1}}^{y^{n-2}},
\end{equation}
where $z_1=\dots=z_p=x$, $z_{p+1}=\dots=z_m=y$, the words $\mathbf{w}_1$, \dots, $\mathbf{w}_p$, $\mathbf{w}'_{p+1}$ are (possibly empty) powers of $x$, and the words $\mathbf{w}''_{p+1}$, $\mathbf{w}_{p+2}$, \dots, $\mathbf{w}_{m+1}$ are (possibly empty) powers of~$y$. Since $p\le m<n-1$, one of the words $\mathbf{w}_1$, \dots, $\mathbf{w}_p$, $\mathbf{w}'_{p+1}$ is nonempty. Deleting a single instance of $x$ from this word yields  a decomposition of $x^{n-2}y^{n-2}$, which is an occurrence of $\mathbf{u}$ in $x^{n-2}y^{n-2}$ with gaps contained in the corresponding gaps of \eqref{eq:decompositionxxy}. By Proposition~\ref{prop:idTn},
$\mT_{n-1}$ satisfies $x^{n-2}y^{n-1}+x^{n-1}y^{n-2}\preccurlyeq x^{n-2}y^{n-2}$.

Conversely, if $\mathbf{u}=z_1\cdots z_m$ with $m<n-1$ occurs as a subword in $x^{n-2}y^{n-2}$, then again, $\mathbf{u}=x^{p}y^{m-p}$ for some $0\le p\le m$. Any occurrence of $\mathbf{u}$ in $x^{n-2}y^{n-2}$ can be written as
\begin{equation}\label{eq:decompositionxy}
x^{n-2}y^{n-2}=\overbrace{\mathbf{w}_1z_1\mathbf{w}_2\cdots\mathbf{w}_pz_p\mathbf{w}'_{p+1}}^{x^{n-2}}\cdot\overbrace{\mathbf{w}''_{p+1}z_{p+1}\mathbf{w}_{p+2}\cdots\mathbf{w}_mz_m\mathbf{w}_{m+1}}^{y^{n-2}},
\end{equation}
where $z_1=\dots=z_p=x$, $z_{p+1}=\dots=z_m=y$, the words $\mathbf{w}_1$, \dots, $\mathbf{w}_p$, $\mathbf{w}'_{p+1}$ are (possibly empty) powers of $x$, and the words $\mathbf{w}''_{p+1}$, $\mathbf{w}_{p+2}$, \dots, $\mathbf{w}_{m+1}$ are (possibly empty) powers of~$y$. If $0<p<m$ or $m<n-1$, then there is a nonempty factor amongst $\mathbf{w}_1$, \dots, $\mathbf{w}_p$, $\mathbf{w}'_{p+1}$, and inserting $x$ after this factor in the right-hand side of \eqref{eq:decompositionxy}, we obtain a decomposition of the word $x^{n-1}y^{n-2}$, which is an occurrence of $\mathbf{u}$ in $x^{n-1}y^{n-2}$ with the same gaps as in~\eqref{eq:decompositionxy}. If $p=m=n-2$, then $\mathbf{u}=x^{n-2}$, and this subword occurs in $x^{n-2}y^{n-2}$ and $x^{n-2}y^{n-1}$ with the same gaps $\underbrace{\varnothing,\dots,\varnothing}_{\text{$n-2$ times}},\{y\}$. If $p=0$, $m=n-2$, then $\mathbf{u}=y^{n-2}$, and this subword occurs in $x^{n-2}y^{n-2}$ and $x^{n-1}y^{n-2}$ with the same gaps $\{x\},\underbrace{\varnothing,\dots,\varnothing}_{\text{$n-2$ times}}$.  By Proposition~\ref{prop:idTn},
$\mT_{n-1}$ satisfies $x^{n-2}y^{n-1}+x^{n-1}y^{n-2}\succcurlyeq x^{n-2}y^{n-2}$. Hence, \eqref{eq:semiring identity in Tn} holds in $\mT_{n-1}$.

Now let $\alpha\in C_{n+1}$ be defined by
\[
i\alpha:=\begin{cases} i+1 &\text{if } i\le n,\\
n+1 &\text{if } i=n+1.
\end{cases}
\]
Then $1\alpha^{n-1}=n$ and $1\alpha^{n}=n+1$, whence $\alpha^{n}\ne\alpha^{n-1}$. Therefore, the identity \eqref{eq:semigroup identity in Tn} fails in the monoid $C_{n+1}$. As \eqref{eq:semigroup identity in Tn} holds in the monoid $T_{n-1}$ and identities are inherited by submonoids, $C_{n+1}$ does not embed in $T_{n-1}$, and the semiring $\mC^-_{n+1}$ whose multiplicative monoid is isomorphic to $C_{n+1}$ does not embed in $\mT_{n-1}$.

Let $\beta,\gamma\in C_{n}$ be defined by
\[
i\beta:=\begin{cases} i+1 &\text{if } i\le n-1,\\
n &\text{if } i=n;
\end{cases}
\qquad
i\gamma:=\begin{cases} n-1 &\text{if } i\le n-1,\\
n &\text{if } i=n.
\end{cases}
\]
Then $1\beta^{n-2}\gamma^{n-2}=(n-1)\gamma^{n-2}=n-1$, whereas
\[
1\left(\beta^{n-1}\gamma^{n-2}+\beta^{n-2}\gamma^{n-1}\right)=\max\{1\beta^{n-1}\gamma^{n-2},1\beta^{n-2}\gamma^{n-1}\}=\max\{n,n-1\}=n.
\]
Therefore, the identity \eqref{eq:semiring identity in Tn} fails in the \ais{} $\mC_{n}$. As \eqref{eq:semiring identity in Tn} holds in $\mT_{n-1}$ and identities are inherited by subsemirings, $\mC_{n}$  does not embed in $\mT_{n-1}$.
\end{proof}

\begin{remark}
A \emph{divisor} of a semigroup [semiring] is defined as a homomorphic image of one of its subsemigroups [subsemirings]. Since identities are inherited by divisors, the above proof of Proposition \ref{prop:optimality} shows that the result remains valid in the stronger form with embedding replaced by division.
\end{remark}

\subsection{\texorpdfstring{Identities of $\mC^-_n$}{Identities of Cn}} We begin with the following observation, which allows us to extract certain identities of $\mC^-_n$ from those of $\mC_n$.

\begin{lemma}
\label{lem:upsidedown}
For any words $\mathbf u,\mathbf v$, the identity $\mathbf u\preccurlyeq\mathbf v$ holds in the semiring $\mC^-_n$ if and only if the identity $\mathbf u\succcurlyeq\mathbf v$ holds in the semiring $\mC_n$.
\end{lemma}

\begin{proof}
Recall that $O_n$ stands for the monoid of all order-preserving transformations of $[n]$. The map $\overline{\cdot}\colon O_{n}\to O_{n}$, where, for $\alpha\in O_n$, the transformation $\overline{\alpha}$ is defined by
\[
i\overline{\alpha} := n+1-(n+1-i)\alpha \quad \text{for each } i=1,\dots,n,
\]
is an automorphism of the monoid $O_n$. This automorphism reverses order \eqref{eq:semilattice order}: $\alpha\leqslant\beta$ if and only if $\overline{\alpha}\geqslant\overline{\beta}$ for all $\alpha,\beta\in O_n$. Indeed, $\alpha\leqslant\beta$ means that $i\alpha\le i\beta$ for each $i=1,\dots,n$. Then $(n+1-i)\alpha\le(n+1-i)\beta$, and hence
\[
i\overline{\alpha}= n+1-(n+1-i)\alpha \ge  n+1-(n+1-i)\beta = i\overline{\beta} \quad \text{for each } i=1,\dots,n.
\]
Combining the order-reversing property with the obvious fact that the identity transformation $\varepsilon$ satisfies $\varepsilon=\overline{\varepsilon}$, we obtain that the automorphism $\overline{\cdot}$ maps $C_n=\{\alpha\in O_n\mid \alpha\geqslant\varepsilon\}$ onto $C^-_n=\{\alpha\in O_n\mid \alpha\leqslant\varepsilon\}$ and vice versa.

Now suppose that the identity $\mathbf u\succcurlyeq\mathbf v$ holds in $\mC_n$. Since for an arbitrary substitution $\varphi\colon\alf(\mathbf{uv})\to\mC^-_n$, the map $x\mapsto\overline{x\varphi}$ is a substitution $\alf(\mathbf{uv})\to\mC_n$, we have $\overline{\mathbf u\varphi}\geqslant\overline{\mathbf v\varphi}$, and hence $\mathbf u\varphi\leqslant\mathbf v\varphi$. This shows that the identity $\mathbf u\preccurlyeq\mathbf v$ holds in the semiring $\mC^-_n$. The converse follows by symmetry.
\end{proof}

We remark that no analog of Lemma~\ref{lem:upsidedown} holds for inequalities between polynomials because the monoid automorphism $\alpha\mapsto\overline{\alpha}$ does not preserve addition.

Since identities~\eqref{eq:xy>x} and~\eqref{eq:yx>x} hold in the semiring $\mC_n$, Lemma~\ref{lem:upsidedown} gives the following.
the semiring $\mC^-_n$ satisfies the identities
\begin{lemma}
\label{lem:xle1}
The semiring $\mC^-_n$ satisfies the identities
\begin{gather}
\label{eq:xy<x} xy\preccurlyeq x, \\
\label{eq:yx<x} yx\preccurlyeq x.
\end{gather}
\end{lemma}
Of course, this can also be proved directly.

By Lemma~\ref{lem:upsidedown}, reversing all inequalities in the proof of Lemma~\ref{lem:subword} yields a proof of the following dual of the latter lemma.
\begin{corollary}
\label{cor:subword}
Every identity $\mathbf{v}\preccurlyeq\mathbf{u}$, where $\mathbf{u}$ is a subword of $\mathbf{v}$, is a consequence of the identities \eqref{eq:xy<x} and \eqref{eq:yx<x}.
\end{corollary}

Combining Lemma~\ref{lem:xle1} and Corollary~\ref{cor:subword} leads to the following.
\begin{corollary}
\label{cor:subwordinC-}
For any words $\mathbf{u}$, $\mathbf{v}$ such that $\mathbf{u}$ is a subword of $\mathbf{v}$, the identity $\mathbf{v}\preccurlyeq\mathbf{u}$ holds in the semiring $\mC^-_n$.
\end{corollary}

Next, we classify identities with a linear lower side that hold in the semiring $\mC^-_n$.

\begin{proposition}
\label{prop:idCn-}
An identity $x_1\cdots x_k\preccurlyeq\mathbf V$ with a linear lower side holds in the semiring $\mC^-_n$ if and only if, for each $m<n-1$ and $1\le i_1<\cdots<i_m\le k$, there exists a word $\mathbf v\in\mathbf V$ that can be decomposed as
\begin{equation}\label{eq:decomposition}
\mathbf v=\mathbf v_1x_{i_1}^{e_1}\mathbf v_2\cdots \mathbf v_mx_{i_m}^{e_m}\mathbf v_{m+1},
\end{equation}
where $e_1,\dots,e_m\in\{0,1\}$ and the \textup(possibly empty\textup) words $\mathbf{v}_1, \mathbf{v}_2, \dots,  \mathbf{v}_m, \mathbf{v}_{m+1}$ are such that $\alf(\mathbf v_1)\subseteq\{x_1,\dots, x_{i_1{-}1}\}, \alf(\mathbf v_2)\subseteq\{x_{i_1+1},\dots,x_{i_2{-}1}\}, \dots, \alf(\mathbf v_m)\subseteq\{x_{i_{m{-}1}+1},\dots,x_{i_m{-}1}\}$, $\alf(\mathbf v_{m{+}1})\subseteq\{x_{i_m{+}1},\dots, x_{i_k}\}$.
\end{proposition}

Below, for a word of form~\eqref{eq:decomposition}, we refer to the tuple $(m;i_1,\dots,i_m)$ as its \emph{marker tuple}.

\begin{proof}
\textit{Necessity}. Take an identity $x_1\cdots x_k\preccurlyeq\mathbf V$ with a linear lower side that holds in the ai-semiring $\mC^-_n$ and consider arbitrary $m < n-1$ and $1\le i_1<\cdots<i_m\le k$. Let $i_0:=0$ and $i_{m+1}=k+1$. Define a substitution
\[
\varphi\colon\{x_1,\dots,x_k\}\cup\alf(\mathbf V)\to\mC^-_n
\]
by the following rule, where $\alpha\in\mC^-_n$ is represented as the string $1\alpha 2\alpha\cdots n\alpha$:
\begin{equation}\label{eq:substitution}
y\varphi :=
\begin{cases}
\underbrace{1\dots1}_{n-r}\,\underbrace{n-r\dots n-r}_{r} & \text{if } y=x_{i_r},\ r=1,2,\dots,m,\\[1ex]
\underbrace{1\dots1}_{n-r-1}\,\underbrace{n-r\dots n-r}_{r+1} & \text{if } y\in\{x_{i_r+1},\dots,x_{i_{r+1}-1}\},\ r=0,1,\dots,m,\\[1ex]
\underbrace{1\cdots1}_{n} & \text{otherwise}.
\end{cases}
\end{equation}
The substitution $\varphi$ is well-defined since the variables $x_1,\dots,x_k$ are all distinct.

By construction, $n((x_1\cdots x_k)\varphi)=n-m>1$. Because of the inequality $x_1\cdots x_k\preccurlyeq\mathbf V$, this implies $n(\mathbf V\varphi)>1$. Since $n(\mathbf V\varphi)=\max_{\mathbf v\in\mathbf V}\{n(\mathbf v\varphi)\}$, there must exist a word $\mathbf v\in\mathbf V$ such that $n(\mathbf v\varphi)>1$. We analyze which variables occur in such a word $\mathbf v$, and in which order. Our analysis is organized as a sequence of claims whose proofs follow the same pattern. We argue by contraposition: if the claim is false, we provide a subword $\mathbf u$ of $\mathbf v$ such that $n(\mathbf u\varphi)=1$. Since the identity $\mathbf{v}\preccurlyeq\mathbf{u}$ holds in $\mC^-_n$ by Corollary~\ref{cor:subword}, we then have $1<n(\mathbf v\varphi)\le n(\mathbf u\varphi)=1$, which is a contradiction.

\begin{claim}
No variable $y\notin\{x_1,\dots,x_k\}$ occurs in $\mathbf v$.
\end{claim}

\begin{proof}
If $y\notin\{x_1,\dots,x_k\}$ occurs in $\mathbf v$, we can set $\mathbf{u}:=y$, since the final case in \eqref{eq:substitution} gives $n(y\varphi)=1$.
\end{proof}

\begin{claim}
No variable $x_{i_r}$, $r=1,2,\dots,m$, occurs in $\mathbf v$ more than once.
\end{claim}

\begin{proof}
If $x_{i_r}$ occurs in $\mathbf v$ more than once, we can set $\mathbf{u}:=x_{i_r}^2$, since the first case in \eqref{eq:substitution} gives $n((x_{i_r}^2)\varphi)=1$.
\end{proof}

\begin{claim}
If $1\le a<b\le m$, then no occurrence of $x_{i_b}$ precedes an occurrence of $x_{i_a}$ in $\mathbf v$.
\end{claim}

\begin{proof}
If an occurrence of $x_{i_b}$ precedes an occurrence of $x_{i_a}$ in $\mathbf v$, we set $\mathbf{u}:=x_{i_b}x_{i_a}$. The first case in \eqref{eq:substitution} gives $r(x_{i_a}\varphi)=1$ for each $r=1,\dots,n-a$ and $n(x_{i_b}\varphi)=n-b$. Then $n((x_{i_b}x_{i_a})\varphi)=n(x_{i_b}\varphi)(x_{i_a}\varphi)=(n-b)(x_{i_a}\varphi)=1$, because $a<b$
\end{proof}

\begin{claim}
If $1\le a<b\le m+1$, then no occurrence of $x_t$ with $t\in\{{i_{b-1}+1},\dots,{i_b-1}\}$ precedes an occurrence of $x_{i_a}$ in $\mathbf v$.
\end{claim}

\begin{proof}
If an occurrence of $x_t$ precedes an occurrence of $x_{i_a}$ in $\mathbf v$, we set $\mathbf{u}:=x_tx_{i_a}$. Then $r(x_{i_a}\varphi)=1$ for $1\le r\le n-a$, and $n(x_t\varphi)=n-b+1$ by the first and second cases in~\eqref{eq:substitution}. Therefore, $n((x_tx_{i_a})\varphi)=n(x_t\varphi)(x_{i_a}\varphi)=(n-b+1)(x_{i_a}\varphi)=1$, because $a<b$.
\end{proof}

\begin{claim}
If $1\le a\le b\le m$, then no occurrence of $x_{i_b}$ precedes an occurrence of $x_s$ with $s\in\{i_{a-1}+1,\dots,i_a-1\}$ in $\mathbf v$.
\end{claim}

\begin{proof}
If an occurrence of $x_{i_b}$ precedes an occurrence of $x_s$ in $\mathbf v$, we set $\mathbf{u}:=x_{i_b}x_s$. Then $r(x_s\varphi)=1$ for $1\le r\le n-a$ and $n(x_{i_b}\varphi)=n-b$ by the second and first cases in~\eqref{eq:substitution}. Therefore, $n((x_{i_b}x_s)\varphi)=n(x_{i_b}\varphi) (x_s\varphi)=(n-b)(x_s\varphi)=1$, because $a\le b$.
\end{proof}

\begin{claim}
If $1\le a<b\le m+1$, then no occurrence of $x_t$ with $t\in\{{i_{b-1}+1},\dots,{i_b-1}\}$ precedes an occurrence of $x_s$ with $s\in\{i_{a-1}+1,\dots,i_a-1\}$ in $\mathbf v$.
\end{claim}

\begin{proof}
If an occurrence of $x_t$ precedes an occurrence of $x_s$ in $\mathbf v$, we set $\mathbf{u}:=x_tx_s$. Then $r(x_s\varphi)=1$ for each $1\le r\le n-a$ and $n(x_t\varphi)=n-b$ by the second case in~\eqref{eq:substitution}. Therefore, $n((x_tx_s)\varphi)=n(x_t\varphi)(x_s\varphi)=(n-b)(x_s\varphi)=1$, because $a<b$.
\end{proof}

Now it is easy to see that, to satisfy the restrictions on the variables and their order in Claims~1--6, the word $\mathbf v$ must have form~\eqref{eq:decomposition} with marker tuple $(m;i_1,\dots,i_m)$.

\smallskip

\textit{Sufficiency}. For arbitrary $m<n-1$ and $1\le i_1<\cdots<i_m\le k$, choose a summand of the polynomial $\mathbf V$ that has form~\eqref{eq:decomposition} with marker tuple $(m;i_1,\dots,i_m)$ and denote this summand by $\mathbf v_{i_1,\dots,i_m}$.  Then we can write $\mathbf V$ as
\[
\mathbf V=\sum_{m=1}^{n-2}\left(\sum_{1\le i_1<\cdots<i_m\le k}\mathbf v_{i_1,\dots,i_m}\right)+\mathbf W
\]
for some polynomial $\mathbf W$. (It is possible that the same summand of $\mathbf V$ is chosen for different marker tuples $(m,i_1,\dots,i_m)$, but this is not an issue because of idempotency of addition: the sum
\[
\sum_{m=1}^{n-2}\left(\sum_{1\le i_1<\cdots<i_m\le k}\mathbf v_{i_1,\dots,i_m}\right)\]
may contain as many equal summands as needed.)

If $\mathbf v_{i_1,\dots,i_m}=\mathbf v_1x_{i_1}^{e_1}\mathbf v_2\cdots\mathbf v_mx_{i_m}^{e_m}\mathbf v_{m+1}$ is decomposition~\eqref{eq:decomposition} of the word $\mathbf v_{i_1,\dots,i_m}$, define $\mathbf v'_{i_1,\dots,i_m}:=\mathbf v_1x_{i_1}\mathbf v_2\cdots\mathbf v_mx_{i_m}\mathbf v_{m+1}$. Since $e_1,\dots,e_m\in\{0,1\}$, the word $\mathbf v_{i_1,\dots,i_m}$ is a subword of $\mathbf v'_{i_1,\dots,i_m}$, and by Corollary~\ref{cor:subword}, the identity $\mathbf v'_{i_1,\dots,i_m}\preccurlyeq\mathbf v_{i_1,\dots,i_m}$ holds in the semiring $\mC^-_n$. Summing over all marker tuples $(m,i_1,\dots,i_m)$ with $m<n-1$ and $1\le i_1<\cdots<i_m\le k$, we obtain that $\mC^-_n$ also satisfies the identity
\begin{equation}\label{eq:sum}
\sum_{m=1}^{n-2}\left(\sum_{1\le i_1<\cdots<i_m\le k}\mathbf v'_{i_1,\dots,i_m}\right)\preccurlyeq\sum_{m=1}^{n-2}\left(\sum_{1\le i_1<\cdots<i_m\le k}\mathbf v_{i_1,\dots,i_m}\right)\preccurlyeq\mathbf V.
\end{equation}

Now consider an arbitrary subword of the word $x_1\cdots x_k$ of length less than $n-1$. It has the form $x_{i_1}\cdots x_{i_m}$ for some $m<n-1$ and $1\le i_1<\cdots<i_m\le k$; this subword has a unique occurrence in $x_1\cdots x_k$, with gaps $G_1=\{x_1,\dots, x_{i_1{-}1}\}$, \dots, $G_{m+1}=\{x_{i_m{+}1},\dots, x_{i_k}\}$. (Here we again use the variables $x_1,\dots,x_k$ are all distinct.) By construction, $x_{i_1}\cdots x_{i_m}$ also occurs as a subword in the word $\mathbf v'_{i_1,\dots,i_m}=\mathbf v_1x_{i_1}\mathbf v_2\cdots\mathbf v_mx_{i_m}\mathbf v_{m+1}$ and $\alf(\mathbf v_\ell)\subseteq G_\ell$ for each $\ell=1,\dots,m+1$. We can therefore apply Proposition~\ref{prop:idTn} to conclude that the identity
\begin{equation}\label{eq:Tn}
x_1\cdots x_k \preccurlyeq \sum_{m=1}^{n-2}\left(\sum_{1\le i_1<\cdots<i_m\le k}\mathbf v'_{i_1,\dots,i_m}\right)
\end{equation}
holds in the semiring $\mT_{n-1}$, and hence, by Corollary~\ref{cor:C_(n+1)toT(n)}, it also holds in $\mC^-_n$. Combining \eqref{eq:Tn} and \eqref{eq:sum}, we obtain that
the semiring $\mC^-_n$ satisfies the identity $x_1\cdots x_k \preccurlyeq\mathbf V$.
\end{proof}

We are now in a position to provide a finite identity basis for the semiring $\mC^-_n$.
\begin{theorem}
\label{thm:Cn-}
The identities \eqref{eq:xy<x}, \eqref{eq:yx<x}, and
\[
\sigma_n:\ \ x_1\cdots x_{n-1}\approx \sum_{i=1}^{n-1}x_1\cdots x_{i-1}x_i^{n}x_{i+1}\cdots x_{n-1}
\]
constitute an identity basis for the semiring $\mC^-_{n}$.
\end{theorem}

\begin{proof}
First, we have to verify that the listed identities hold in $\mC^-_{n}$. For \eqref{eq:xy<x} and \eqref{eq:yx<x}, this is shown in Lemma~\ref{lem:xle1}.

Let us turn to $\sigma_n$. Corollary~\ref{cor:subwordinC-} gives, for each $i=1,\dots,n-1$, the inequality
\[
x_1\cdots x_{i-1}x_i^{n}x_{i+1}\cdots x_{n-1}\preccurlyeq x_1\cdots x_{n-1},
\]
because $x_1\cdots x_{n-1}$ is a subword of the word $x_1\cdots x_{i-1}x_i^{n}x_{i+1}\cdots x_{n-1}$. Summing all these inequalities yields
\begin{equation}\label{eq:sum<product}
\sum_{i=1}^{n-1}x_1\cdots x_{i-1}x_i^{n}x_{i+1}\cdots x_{n-1}\preccurlyeq x_1\cdots x_{n-1},
\end{equation}
because of idempotency of addition. The opposite inequality
\begin{equation}\label{eq:product<sum}
x_1\cdots x_{n-1}\preccurlyeq\sum_{i=1}^{n-1}x_1\cdots x_{i-1}x_i^{n}x_{i+1}\cdots x_{n-1}
\end{equation}
holds even in the semiring $\mT_{n-1}$. Indeed, if a word $x_{i_1}\cdots x_{i_m}$ with $m<n-1$ occurs as a subword in the word $x_1\cdots x_{n-1}$, then $i_0:=0<i_1<\dots<i_m<n=:i_{m+1}$ and there exists an index $i$ such that $i_j<i<i_{j+1}$ for some $j=0,1,\dots,m$. Clearly, $x_{i_1}\cdots x_{i_m}$ occurs as a subword in the word $x_1\cdots x_{i-1}x_i^{n}x_{i+1}\cdots x_{n-1}$, and the gaps of its occurrence in that word are the same as those in $x_1\cdots x_{n-1}$. Therefore, Proposition~\ref{prop:idTn} applies. By Corollary~\ref{cor:C_(n+1)toT(n)}, inequality \eqref{eq:product<sum} holds in $\mC^-_{n}$ and, together with \eqref{eq:sum<product}, gives the identity $\sigma_n$ in $\mC^-_{n}$.

Now we should prove that every identity that holds in $\mC^-_{n}$ can be deduced from the set
\[
\Delta_n:=\{xy\preccurlyeq x,\ yx\preccurlyeq x,\ \sigma_n\}.
\]
We proceed in four steps. First, we deduce some useful auxiliary identities from $\Delta_n$.

\begin{step}
$\Delta_n$ implies the identity $x^{n-1}\approx x^n.$
\end{step}

\begin{proof}
Substituting $x$ for every variable in $\sigma_n$ and applying idempotency of addition, we obtain the identity $x^{n-1}\approx x^{2n-1}$. Clearly, $x^{n-1}$ is a subword of the word $x^n$, which in turn is a subword of $x^{2n-1}$. By Corollary~\ref{cor:subword}, the inequalities $x^{2n-1}\preccurlyeq x^n\preccurlyeq x^{n-1}$ follow from \eqref{eq:xy<x} and \eqref{eq:yx<x}. Hence, $x^{n-1}\approx x^n$ follows from $\Delta_n$.
\end{proof}

Define $\mathbf D_{k}$ as ${\displaystyle\sum\limits_{m=1}^{n-2}}\biggl({\displaystyle\sum\limits_{1\le i_1<\cdots<i_{m}\le k}}(x_1\cdots x_{i_1-1})^{n}x_{i_1}(x_{i_1+1}\cdots x_{i_2-1})^{n}\cdots x_{i_{m}}(x_{i_{m}+1}\cdots x_{k})^{n}\biggr)$,\linebreak where $x_1,\dots,x_k$ are distinct variables. Here and below, we adopt the convention that any product of the form $x_a\cdots x_b$ with $a>b$ represents the empty word. Note that all summands of $\mathbf D_k$ are words of form~\eqref{eq:decomposition}, so the concept of a marker tuple makes sense for them.

\begin{step}
$\Delta_n$ implies the identity $\delta_{n-1}^{n}:\, x_1x_2\cdots x_{n-1}\preccurlyeq\mathbf D_{n-1}$.
\end{step}

\begin{proof}
The $i$-th summand $x_1\cdots x_{i-1}x_i^{n}x_{i+1}\cdots x_{n-1}$ on the right-hand side of $\sigma_n$ appears in $\mathbf D_{n-1}$ as the summand with marker tuple $(n-2;1,\dots,i-1,i+1,\dots,n-1)$.
Therefore $\mathbf D_{n-1}$ is less than or equal to the right-hand side of $\sigma_n$, and hence, $\sigma_n$ implies $\delta_{n-1}^{n}$.
\end{proof}

\begin{step}
$\Delta_n$ implies the identity $\delta_{k}^{n}:\, x_1x_2\cdots x_{k}\preccurlyeq\mathbf D_{k}$ for each $k\ge n-1$.
\end{step}

\begin{proof} We induct on $k$. The base case, $k=n-1$, is established in Step~2. For the inductive step, it remains to verify that $\delta_{k+1}^{n}$ follows from $\Delta_n\cup\{\delta_{k}^{n}\}$ for all $k\ge n-1$.

We divide the summands of $\mathbf D_{k}$ into three groups according to their marker tuples:
\begin{itemize}
  \item[I)] the summands with marker tuples $(m;i_1,\dots,i_m)$ such that $m\le n-2$ and $i_m<k$;
  \item[II)] the summands with marker tuples $(m;i_1,\dots,i_m)$ such that $m\le n-3$ and $i_m=k$;
  \item[III)] the summands with marker tuples $(n-2;i_1,\dots,i_{n-2})$ such that $i_{n-2}=k$.
\end{itemize}
Substituting $x_kx_{k+1}$ for $x_k$ in $\delta_{k}^{n}$ yields the identity $x_1\cdots x_{k+1}\preccurlyeq\mathbf P+\mathbf Q+\mathbf R$, where $\mathbf P$, $\mathbf Q$, and $\mathbf R$ denote the sums of the words obtained by replacing $x_k$ with $x_kx_{k+1}$ in all summands of groups I), II), and III), respectively. In their explicit form, the polynomials $\mathbf P$, $\mathbf Q$, and $\mathbf R$ are as follows:
\begin{align*}
\mathbf P:=&\sum_{m=1}^{n-2}\biggl( \sum_{1\le i_1<\cdots<i_{m}< k}(x_1\cdots x_{i_1-1})^{n}x_{i_1}(x_{i_1+1}\cdots x_{i_2-1})^{n}\cdots x_{i_{m}}(x_{i_{m}+1}\cdots x_kx_{k+1})^{n}\biggr),\\
\mathbf Q:=&\sum_{m=1}^{n-3}\biggl( \sum_{1\le i_1<\cdots<i_{m-1}< k}(x_1\cdots x_{i_1-1})^{n}x_{i_1}(x_{i_1+1}\cdots x_{i_2-1})^{n}\cdots x_{i_{m-1}}(x_{i_{m-1}+1}\cdots x_{k-1})^{n}x_kx_{k+1}\biggr),\\
\mathbf R:=&\sum_{1\le i_1<\cdots<i_{n-3}< k}(x_1\cdots x_{i_1-1})^{n}x_{i_1}(x_{i_1+1}\cdots x_{i_2-1})^{n}\cdots x_{i_{n-2}}(x_{i_{n-3}+1}\cdots x_{k-1})^{n}x_kx_{k+1}.
\end{align*}
It is easy to see that each word appearing as a summand in $\mathbf P$ or $\mathbf Q$ also appears as such in $\mathbf D_{k+1}$, and hence, $\mathbf P+\mathbf Q\preccurlyeq\mathbf D_{k+1}$.

A generic summand $\mathbf r_{i_1,\dots,i_{n-3},k}$ in $\mathbf R$ can be written as $\mathbf a_1x_{i_1}\mathbf a_2x_{i_2}\cdots \mathbf a_{n-1}x_{i_{n-1}}\mathbf a_{n}$, where $\mathbf a_1:=(x_1\cdots x_{i_1-1})^{n}$, $\mathbf a_2:=(x_{i_1+1}\cdots x_{i_2-1})^{n}$, \dots, $\mathbf a_{n-2}:=(x_{i_{n-3}+1}\cdots x_{k-1})^{n}$, the words $\mathbf a_{n-1}$ and $\mathbf a_{n}$ are empty, $x_{i_{n-2}}:=x_k$, and $x_{i_{n-1}}:=x_{k+1}$. Using the identity $x^{n}\approx x^{2n}$, which is a consequence of identity $x^{n-1}\approx x^n$ established in Step 1, we obtain
\[
\mathbf r_{i_1,\dots,i_{n-3},k}=\mathbf a_1x_{i_1}\mathbf a_2x_{i_2}\cdots \mathbf a_{n-1}x_{i_{n-1}}\mathbf a_{n}\approx(\mathbf a_1x_{i_1}\mathbf a_2)(\mathbf a_2x_{i_2}\mathbf a_{3})\cdots(\mathbf a_{n-1}x_{i_{n-1}}\mathbf a_{n}).
\]
Substituting $\mathbf a_{s}x_{i_{s}}\mathbf a_{s+1}$ for $x_s$ in the identity $\sigma_n$ for each $s=1,\dots,n-1$, we obtain
\begin{multline}
\label{eq:inductionstep}
(\mathbf a_1x_{i_1}\mathbf a_2)(\mathbf a_2x_{i_2}\mathbf a_{3})\cdots(\mathbf a_{n-1}x_{i_{n-1}}\mathbf a_{n})\approx\\
\sum_{s=1}^{n-1}(\mathbf a_1x_{i_1}\mathbf a_{2})\cdots (\mathbf a_{s-1}x_{i_{s-1}}\mathbf a_{s}) (\mathbf a_{s}x_{i_{s}}\mathbf a_{s+1})^{n}(\mathbf a_{s+1}x_{i_{s+1}}\mathbf a_{s+2})\cdots  (\mathbf a_{n-1}x_{i_{n-1}}\mathbf a_{n}).
\end{multline}
Now we show that each summand on the right-hand side of \eqref{eq:inductionstep} has a subword that occurs as a summand in $\mathbf D_{k+1}$.

For $1\le s\le n-3$, the summand
\[
(\mathbf a_1x_{i_1}\mathbf a_{2})\cdots (\mathbf a_{s-1}x_{i_{s-1}}\mathbf a_{s}) (\mathbf a_{s}x_{i_{s}}\mathbf a_{s+1})^{n}(\mathbf a_{s+1}x_{i_{s+1}}\mathbf a_{s+2})\cdots  (\mathbf a_{n-1}x_{i_{n-1}}\mathbf a_{n})
\]
has the word $\mathbf a_1x_{i_1}\cdots\mathbf a_{s-1}x_{i_{s-1}}(x_{i_{s-1}+1}\cdots x_{i_{s+1}-1})^{n}x_{i_{s+1}}\mathbf a_{s+2}\cdots x_{i_{n-1}}\mathbf a_{n}$ as a subword which occurs in $\mathbf D_{k+1}$ as the summand with marker tuple $(n-2;1,\dots,s-1,s+1,\dots,n-1)$.

For $s=n-2$, the summand $(\mathbf a_1x_{i_1}\mathbf a_{2})\cdots(\mathbf a_{n-3}x_{i_{n-3}}\mathbf a_{n-2})(\mathbf a_{n-2}x_{i_{n-2}}\mathbf a_{n-1})^n(\mathbf a_{n-1}x_{i_{n-1}}\mathbf a_{n})$ has the word $\mathbf a_1x_{i_1}\cdots\mathbf a_{n-2}x_{i_{n-2}}(x_{i_{n-2}+1}\cdots x_{k})^n x_{k+1}$ as a subword, since, by definition, the words $\mathbf a_{n-1}$ and $\mathbf a_{n}$ are empty, $x_{i_{n-2}}=x_k$ and $x_{i_{n-1}}=x_{k+1}$. The latter word occurs in $\mathbf D_{k+1}$ as the summand with marker tuple  $(n-2;1,\dots,n-3,n-1)$.

Finally, for $s=n-1$, re-using that $\mathbf a_{n}$ is empty, $x_{i_{n-2}}=x_k$ and $x_{i_{n-1}}=x_{k+1}$, we see that $(\mathbf a_1x_{i_1}\mathbf a_{2})\cdots(\mathbf a_{n-2}x_{i_{n-2}}\mathbf a_{n-1})(\mathbf a_{n-1}x_{i_{n-1}}\mathbf a_{n})^n$ has the word $\mathbf a_1x_{i_1}\cdots\mathbf a_{n-2}x_{i_{n-2}} \mathbf a_{n-1} x_{k} x_{k+1}^n$ as a subword which occurs in $\mathbf D_{k+1}$ as the summand with marker tuple $(n-2;1,\dots,n-2)$.

By Corollary~\ref{cor:subword}, identities \eqref{eq:xy<x} and \eqref{eq:yx<x} imply the identity
\[
\sum_{s=1}^{n-1}(\mathbf a_1x_{i_1}\mathbf a_{2})\cdots (\mathbf a_{s-1}x_{i_{s-1}}\mathbf a_{s}) (\mathbf a_{s}x_{i_{s}}\mathbf a_{s+1})^{n}(\mathbf a_{s+1}x_{i_{s+1}}\mathbf a_{s+2})\cdots  (\mathbf a_{n-1}x_{i_{n-1}}\mathbf a_{n})\preccurlyeq\mathbf D_{k+1}.
\]
Together with \eqref{eq:inductionstep}, this gives $\mathbf r_{i_1,\dots,i_{n-3},k}\preccurlyeq\mathbf D_{k+1}$ for an arbitrary summand $\mathbf r_{i_1,\dots,i_{n-3},k}$ of $\mathbf R$, and hence, $\mathbf R \preccurlyeq\mathbf D_{k+1}$. Therefore, we obtain the inequalities
\[
x_1\cdots x_{k+1}\stackrel{\delta_{k}^{n}}\preccurlyeq \mathbf P+\mathbf Q+\mathbf R\stackrel{\Delta_n}\preccurlyeq\mathbf D_{k+1},
\]
showing that $\delta_{k+1}^{n}$ is a consequence of $\delta_{k}^{n}$ and $\Delta_n$.
\end{proof}

\begin{step}
$\Delta_n$, together with $x^{n-1}\approx x^n$ and the identities $\delta_{k}^{n}$ for $k\ge n-1$, implies all identities that hold in the semiring $\mC^-_{n}$.
\end{step}

\begin{proof}
In view of Lemma~\ref{lem:linearization}, it suffices to prove the claim for identities $x_1\cdots x_k\preccurlyeq\mathbf V$ with a linear lower side. If $k<n-1$, then, by Proposition~\ref{prop:idCn-}, there exists a word $\mathbf v\in\mathbf V$ with marker tuple $(k;1,\dots,k)$. For such $\mathbf v$, all words $\mathbf v_1,\dots,\mathbf v_{k+1}$ in decomposition \eqref{eq:decomposition} are empty, that is, $\mathbf v=x_1^{e_1}\cdots x_k^{e_k}$, where $e_1,\dots,e_k\in\{0,1\}$. Clearly, $\mathbf v$ is then a subword of $x_1\cdots x_k$. By Corollary~\ref{cor:subword}, the identity $x_1\cdots x_k\preccurlyeq\mathbf v$ is a consequence of \eqref{eq:xy<x} and \eqref{eq:yx<x}, and hence so is the identity $x_1\cdots x_k\preccurlyeq\mathbf V$.

Thus, we assume that $k\ge n-1$ for the rest of the proof.

We can represent the polynomial $\mathbf V$ as in the sufficiency proof of Proposition~\ref{prop:idCn-}:
\[
\mathbf V=\sum_{m=1}^{n-2}\left(\sum_{1\le i_1<\cdots<i_m\le k}\mathbf v_{i_1,\dots,i_m}\right)+\mathbf W,
\]
where $\mathbf W$ is a polynomial and
\[
\mathbf v_{i_1,\dots,i_m}=\mathbf v_1x_{i_1}^{e_1}\mathbf v_2\cdots \mathbf v_mx_{i_m}^{e_m}\mathbf v_{m+1},
\]
with $e_1,\dots,e_m\in\{0,1\}$ is a word chosen as the summand with marker tuple $(m;i_1,\dots,i_m)$.

Let $i_0:=0$ and $i_{m+1}=k+1$. The (possibly empty) words $\mathbf{v}_1, \mathbf{v}_2, \dots,  \mathbf{v}_m, \mathbf{v}_{m+1}$ satisfy $\alf(\mathbf v_s)\subseteq\{x_{i_{s-1}+1},\dots,x_{i_s-1}\}$ for each $1\le s\le m+1$. Let $\ell_s:=|\mathbf v_s|$. Then the word $\mathbf v_s$ occurs as a subword in the word $(x_{i_{s{-}1}+1}\cdots x_{i_s{-}1})^{\ell_s}$; indeed, the leftmost variable of $\mathbf v_s$ appears in the first copy of the factor $x_{i_{s{-}1}+1}\cdots x_{i_s{-}1}$, the next variable in the second copy, and so on. Hence, for $\ell:=\max\{n,\ell_1,\dots,\ell_{m+1}\}$, the word $\mathbf v_{i_1,\dots,i_m}$ occurs as a subword in
\[
(x_1\cdots x_{i_1-1})^\ell x_{i_1}(x_{i_1+1}\cdots x_{i_2-1})^\ell \cdots x_{i_{m}}(x_{i_{m}+1}\cdots x_k)^\ell.
\]
By Corollary~\ref{cor:subword}, identities \eqref{eq:xy<x} and \eqref{eq:yx<x} imply the identity
\[
(x_1\cdots x_{i_1-1})^\ell x_{i_1}(x_{i_1+1}\cdots x_{i_2-1})^\ell \cdots x_{i_{m}}(x_{i_{m}+1}\cdots x_k)^\ell\preccurlyeq \mathbf v_{i_1,\dots,i_m}.
\]
Since $\ell\ge n$, we can apply $x^{n-1}\approx x^n$ to replace $\ell$ with $n$ on the lower side of the last inequality, thus obtaining the identity
\[
(x_1\cdots x_{i_1-1})^n x_{i_1}(x_{i_1+1}\cdots x_{i_2-1})^n \cdots x_{i_{m}}(x_{i_{m}+1}\cdots x_k)^n\preccurlyeq \mathbf v_{i_1,\dots,i_m}.
\]
Summing over all marker tuples $(m;i_1,\dots,i_m)$ with $m<n-1$ and $1\le i_1<\cdots<i_m\le k$ yields the identity
\[
\mathbf D_{k}\preccurlyeq\sum_{m=1}^{n-2}\left(\sum_{1\le i_1<\cdots<i_m\le k}\mathbf v_{i_1,\dots,i_m}\right),
\]
since, by definition,
\[
\mathbf D_{k}:=\sum_{m=1}^{n-2}\biggl( \sum_{1\le i_1<\cdots<i_{m}\le k}(x_1\cdots x_{i_1-1})^{n}x_{i_1}(x_{i_1+1}\cdots x_{i_2-1})^{n}\cdots x_{i_{m}}(x_{i_{m}+1}\cdots x_{k})^{n}\biggr)
\]
is the upper side of the identity $\delta_{k}^{n}$. Since $k\ge n-1$, we can use $\delta_{k}^{n}$ to deduce
\[
x_1x_2\cdots x_{k}\stackrel{\delta_{k}^{n}}\preccurlyeq\mathbf D_{k}\preccurlyeq\sum_{m=1}^{n-2}\left(\sum_{1\le i_1<\cdots<i_m\le k}\mathbf v_{i_1,\dots,i_m}\right)\preccurlyeq\sum_{m=1}^{n-2}\left(\sum_{1\le i_1<\cdots<i_m\le k}\mathbf v_{i_1,\dots,i_m}\right)+\mathbf W=\mathbf V,
\]
as required.
\end{proof}

Since $\Delta_n$ implies the identities $x^{n-1}\approx x^n$ and $\delta_{k}^{n}$ with $k\ge n-1$, as shown in Steps 1--3, Step~4 establishes Theorem~\ref{thm:Cn-}.
\end{proof}

\subsection{Identities shared by two Catalan semirings}
While the semirings $\mC_n$ and $\mC^-_m$ are individually finitely based, as shown in Theorems~\ref{thm:basisCn} and~\ref{thm:Cn-}, it turns out that their common identities do not admit a finite basis for any $n,m\ge 3$.

We will use the concept of a variety, which we now recall. The class of all \ais{}s that satisfy every identity in a given set $\Sigma$ is called the \emph{variety defined by $\Sigma$}.  It is easy to see that the satisfaction of an identity is inherited by forming direct products and taking homomorphic images and subsemirings, so that each variety is closed under these three operators. Varieties can be characterized by this closure property (the HSP-theorem; see \cite[Theorem II.11.9]{BuSa81}).

A variety is \emph{\fb} if it can be defined by a finite set of identities; otherwise, it is \emph{\nfb}. For a class $\mathfrak{S}$ of \ais{}s, the \emph{variety generated by $\mathfrak{S}$} is the variety defined by the set of all identities holding in each \ais{} from $\mathfrak{S}$. We denote this variety by $\var\mathfrak{S}$ and write $\var\mS$ rather than $\var\{\mS\}$ for classes consisting of a single \ais{} $\mS$. Thus, $\var\mS$ is  the variety defined by $\mathsf{Id}(\mS)$. By definition, $\mS$ and $\var\mS$ are simultaneously finitely or \nfb.

\begin{proposition}
\label{prop:nfb-general}
Let $\mathfrak{T}$ be the variety defined by the identities
\[
\tau_k:\,x_1x_2\cdots x_{2k+1}\preccurlyeq z_0x_1z_1x_2\cdots x_{2k+1}z_{2k+1}+x_1\cdots x_k\cdot x_{k+2}\cdots x_{2k+1}, \ k=2,3,\dotsc.
\]
Every subvariety of\/ $\mathfrak{T}$ containing the semirings $\mC_3$ and $\mC^-_3$ is nonfinitely based.
\end{proposition}

\begin{proof}
Arguing by contradiction, assume that some subvariety of $\mathfrak{T}$ containing  $\mC_3$ and $\mC^-_3$ admits a finite identity basis $\Sigma$. By Lemma~\ref{lem:linearization}, we may assume that $\Sigma$ consists of identities with a linear lower side. Fix an integer $m$ strictly greater than the number of variables occurring in the identities of $\Sigma$, and let $N:=m+1$.

Let $\mathbf U:=x_1x_2\cdots x_{2N+1}+z_0x_1z_1x_2\cdots x_{2N+1}z_{2N+1}+x_1\cdots x_N\cdot x_{N+2}\cdots x_{2N+1}$; then the identity $\tau_N$ can be rewritten as
\[
\mathbf U\approx z_0x_1z_1x_2\cdots x_{2N+1}z_{2N+1}+x_1\cdots x_N\cdot x_{N+2}\cdots x_{2N+1}.
\]
Specializing Birkhoff's completeness theorem for equational logic \cite[Theorem~14.17]{BuSa81} to \ais{}s shows that $\tau_N$ can be deduced from $\Sigma$ if and only if there exists a sequence
\begin{equation}\label{eq:deduction}
\mathbf U=\mathbf U_0\Rightarrow\mathbf U_1\Rightarrow\dots\Rightarrow\mathbf U_d=z_0x_1z_1x_2\cdots x_{2N+1}z_{2N+1}+x_1\cdots x_N\cdot x_{N+2}\cdots x_{2N+1}
\end{equation}
such that, for each $i=1,\dots,d$, there exist an identity $\mathbf P_i\approx\mathbf Q_i$ in $\Sigma$, a substitution $\varphi_i$ from $\alf(\mathbf{P}_i\mathbf{Q}_i)$ to the set of all proper polynomials in the variables from the set $\alf(\mathbf{U}_{i-1}\mathbf U_i)$, (not necessarily proper) polynomials $\mathbf A_i,\mathbf B_i$, and (either proper or empty) polynomial $\mathbf C_i$ such that either
\begin{align*}
\mathbf U_{i-1}&=\mathbf A_i(\mathbf P_i\varphi_i)\mathbf B_i+\mathbf C_i\ \text{ and }\ \mathbf U_{i}=\mathbf A_i(\mathbf Q_i\varphi_i)\mathbf B_i+\mathbf C_i,\ \text{or}\\
\mathbf U_{i-1}&=\mathbf A_i(\mathbf Q_i\varphi_i)\mathbf B_i+\mathbf C_i\ \text{ and }\ \mathbf U_{i}=\mathbf A_i(\mathbf P_i\varphi_i)\mathbf B_i+\mathbf C_i.
\end{align*}
Clearly, there exists a deduction \eqref{eq:deduction} with $\mathbf U\ne\mathbf U_1$. Fix such a deduction, and let $x_1\cdots x_n\preccurlyeq\mathbf V$ be the identity in $\Sigma$ that is applied at the first step $\mathbf U\Rightarrow\mathbf U_1$ of \eqref{eq:deduction}. (Here $x_1,\dots,x_n$ are distinct variables and $n<m$ according to our assumptions about the basis $\Sigma$.) Rewrite this identity as
\begin{equation}\label{eq:linear}
x_1\cdots x_n+\mathbf V\approx \mathbf V.
\end{equation}
Then there exist a substitution $\varphi$ from $\{x_1,\dots,x_n\}\cup\alf(\mathbf V)$ to the set of all polynomials in the variables $x_1,\dots,x_{2N+1},z_0,\dots,z_{2N+1}$, (not necessarily proper) polynomials $\mathbf A,\mathbf B$, and (either proper or empty) polynomial $\mathbf C$ such that one of the following holds:
\begin{equation}
\label{eq:subst}
\begin{split}
\mathbf U=\mathbf A((x_1\cdots x_n+\mathbf V)\varphi)\mathbf B+\mathbf C\ \text{ and } \ \mathbf U_1=\mathbf A(\mathbf V\varphi)\mathbf B+\mathbf C,\\
\mathbf U=\mathbf A(\mathbf V\varphi)\mathbf B+\mathbf C\ \text{ and } \ \mathbf U_1=\mathbf A((x_1\cdots x_n+\mathbf V)\varphi)\mathbf B+\mathbf C.
\end{split}
\end{equation}
We will prove that for either case in~\eqref{eq:subst} to hold it is necessary that $(x_1\cdots x_n)\varphi=\mathbf u\varphi$ for some $\mathbf u\in\mathbf V$. Then $(x_1\cdots x_n+\mathbf V)\varphi=\mathbf V\varphi$ because of idempotency of addition, and hence each case in \eqref{eq:subst} leads to $\mathbf U=\mathbf U_1$, which is the contradiction we seek.

Since identity \eqref{eq:linear} holds in the semiring $\mC_3$, it follows from Proposition~\ref{prop:idCn} that, for any $1\le i<j\le n$, the word $x_ix_j$, which occurs as a subword $x_1\cdots x_n$, also occurs as a subword in some word $\mathbf v_{i,j}\in\mathbf V$. For each $1\le i<j\le n$, fix such a word $\mathbf v_{i,j}$. For later use, it is convenient to extend the notation $\mathbf v_{i,j}$ to all pairs of distinct indices $i,j\in\{1,\dots,n\}$ by setting $\mathbf v_{i,j}:=\mathbf v_{s,t}$, where $s=\min\{i,j\}$ and $t=\max\{i,j\}$. Further, since \eqref{eq:linear} holds in the semiring $\mC^-_3$, Proposition~\ref{prop:idCn-} implies that, for each $1\le i\le n$, there is a word $\mathbf w_i\in\mathbf V$ of the form $\mathbf w'x_i^{e_i}\mathbf w''$, where $e_i\in\{0,1\}$, $\alf(\mathbf w')\subseteq\{x_1,\dots,x_{i-1}\}$, and $\alf(\mathbf w'')\subseteq\{x_{i+1},\dots,x_n\}$. For each $1\le i\le n$, fix such a word $\mathbf w_i$ and put $\mathbf W:=\{\mathbf w_i\mid 1\le i\le n\}$.

Observe that each summand of the polynomial $\mathbf U$ is a linear word.

Take arbitrary $\mathbf a\in \mathbf A$, $\mathbf b\in\mathbf B$ and $\mathbf v\in\mathbf W$. By~\eqref{eq:subst}, we have $\mathbf a(\mathbf v\varphi)\mathbf b\in\mathbf U$, and so $\mathbf a(\mathbf v\varphi)\mathbf b$ is a linear word. Then the word $\mathbf v$ must also be linear. We have $\alf(\mathbf v)\subseteq\{x_1,\dots,x_n\}$, since $\mathbf v\in\mathbf W$, and therefore there are $s\le n$ and distinct $1\le i_1,\dots,i_s\le n$ such that $\mathbf v=x_{i_1}\cdots x_{i_s}$. Suppose that $i_p>i_q$ for some $p<q$. Take $y\in \alf(x_{i_q}\varphi)$ and $z\in \alf(x_{i_p}\varphi)$. The variable $y$ precedes the variable $z$ in the linear word $\mathbf a(\mathbf v_{i_q,i_p}\varphi)\mathbf b$, while the variable $z$ precedes the variable $y$ in the linear word $\mathbf a(\mathbf v\varphi)\mathbf b$. However, this is impossible because both $\mathbf a(\mathbf v\varphi)\mathbf b$ and $\mathbf a(\mathbf v_{i_q,i_p}\varphi)\mathbf b$ are summands of $\mathbf U$ by~\eqref{eq:subst}. Therefore, $1\le i_1<\dots<i_s\le n$, that is, $\mathbf v$ is a subword of $x_1\cdots x_n$.

If $\mathbf v=x_1\cdots x_n$, then $(x_1\cdots x_n)\varphi=\mathbf v\varphi$, and we are done. So, let $\mathbf v$ be a proper subword of $x_1\cdots x_n$. Take an arbitrary $r\in\{1,\dots,n\}\setminus\{i_1,\dots,i_s\}$ and an arbitrary $j\in\{i_1,\dots,i_s\}$. Since $\mathbf a(\mathbf v_{r,j}\varphi)\mathbf b\in\mathbf U$ by~\eqref{eq:subst} and each summand in $\mathbf U$ is linear, we have $\alf(x_r\varphi)\cap\alf(x_j\varphi)=\varnothing$, and hence $\alf(x_r\varphi)\cap\alf(\mathbf v\varphi)=\varnothing$.

Consider two cases, depending on which summand of $\mathbf U$ coincides with $\mathbf a(\mathbf v\varphi)\mathbf b$.

\smallskip

\textbf{Case 1}: $\mathbf a(\mathbf v\varphi)\mathbf b=z_0x_1z_1\cdots x_{2N+1}z_{2N+1}$.

If the word $\mathbf a(\mathbf v_{r,i_1}\varphi)\mathbf b$ also equals $z_0x_1z_1\cdots x_{2N+1}z_{2N+1}$, then $\alf(x_r\varphi)\subseteq\alf(\mathbf a\mathbf b)$, since $\alf(\mathbf a\mathbf b)\cup\alf(\mathbf v\varphi)=\{z_0,\dots,z_{2N+1},x_1,\dots,x_{2N+1}\}$ and $\alf(x_r\varphi)\cap\alf(\mathbf v\varphi)=\varnothing$. However, this contradicts the fact that $\mathbf a(\mathbf v_{r,i_1}\varphi)\mathbf b$ is a linear word  being a summand of $\mathbf U$ by~\eqref{eq:subst}. Therefore, $\mathbf a(\mathbf v_{r,i_1}\varphi)\mathbf b$ equals one of the other two summands of $\mathbf U$, but then
\[
\alf(\mathbf a(\mathbf v_{r,i_1}\varphi)\mathbf b)\subseteq\{x_1,\dots x_{2N+1}\}.
\]
This is only possible if the words $\mathbf a$ and $\mathbf b$ are empty, because, under the premise of Case~1, if $\mathbf a$ is nonempty, it starts with the variable $z_0$, and if $\mathbf b$ is nonempty, it ends with the variable $z_{2N+1}$. Then the variable $z_0$ appears in $x_{i_1}\varphi$, contradicting the fact that $x_{i_1}\varphi$ is a factor of $\mathbf v_{r,i_1}\varphi$ and $\alf(\mathbf v_{r,i_1}\varphi)\subseteq\{x_1,\dots,x_{2N+1}\}$. This shows that Case~1 is impossible.

\smallskip

\textbf{Case 2}: $\mathbf a(\mathbf v\varphi)\mathbf b=x_1\cdots x_{2N+1}$ or $\mathbf a(\mathbf v\varphi)\mathbf b=x_1\cdots x_N\cdot x_{N+2}\cdots x_{2N+1}$.

Suppose that there are two distinct $r_1,r_2\in\{1,\dots,n\}\setminus\{i_1,\dots,i_s\}$. The word $\mathbf a(\mathbf v_{r_1,r_2}\varphi)\mathbf b$ is also a summand of $\mathbf U$. If this summand is either $x_1\cdots x_{2N+1}$ or $x_1\cdots x_N\cdot x_{N+2}\cdots x_{2N+1}$, the words $\mathbf a(\mathbf v\varphi)\mathbf b$ and $\mathbf a(\mathbf v_{r_1,r_2}\varphi)\mathbf b$ can differ only in that one has an occurrence of the variable $x_{N+1}$ while the other does not. Then only one of the sets $\alf(x_{r_1}\varphi)$ and $\alf(x_{r_2}\varphi)$ can be disjoint from $\alf(\mathbf v\varphi)$, and we have already shown that $\alf(x_r\varphi)\cap\alf(\mathbf v\varphi)=\varnothing$ for an arbitrary $r\in\{1,\dots,n\}\setminus\{i_1,\dots,i_s\}$. Hence $\mathbf a(\mathbf v_{r_1,r_2}\varphi)\mathbf b=z_0x_1z_1\cdots x_{2N+1}z_{2N+1}$. Under the premise of Case~2, this equality implies that the words $\mathbf a$ and $\mathbf b$ are empty. Therefore, $\alf(\mathbf v\varphi)\subseteq\{x_1,\dots x_{2N+1}\}$, and hence, both $\alf(x_{r_1}\varphi)$ and $\alf(x_{r_2}\varphi)$ must be contained
in the set $\{z_0,\dots,z_{2N+1},x_N\}$. Since $\mathbf v_{r_1,r_2}\varphi$ is a linear word, the sets $\alf(x_{r_1}\varphi)$ and $\alf(x_{r_2}\varphi)$ cannot both equal $\{x_N\}$. Thus, at least one of the words $x_{r_1}\varphi$ or $x_{r_2}\varphi$ contains an occurrence of a variable from $\{z_0,\dots,z_{2N+1}\}$. For certainty, assume that the word $x_{r_1}\varphi$ does so.

Since $\mathbf a$ and $\mathbf b$ are empty, we have $\mathbf v\varphi=x_1\cdots x_{2N+1}$ or $\mathbf v\varphi=x_1\cdots x_N\cdot x_{N+2}\cdots x_{2N+1}$. In either case, since $N=m+1$ and $\mathbf v$ contains fewer than $m$ variables, there must exist an index $t\in\{i_1,\dots,i_s\}$ such that the word $x_t\varphi$ has a factor of the form $x_dx_{d+1}$. The word $\mathbf v_{r_1,t}\varphi$, which is a summand of $\mathbf U$, contains both $x_{r_1}\varphi$ and $x_t\varphi$ as factors. Since the word $x_{r_1}\varphi$ contains an occurrence of a variable from $\{z_0,\dots,z_{2N+1}\}$, so does the word $\mathbf v_{r_1,t}\varphi$, which therefore can be nothing but $z_0x_1z_1\cdots x_{2N+1}z_{2N+1}$. However, the latter word has no factor of the form $x_dx_{d+1}$, a contradiction. We conclude that the set $\{1,\dots,n\}\setminus\{i_1,\dots,i_s\}$ contains exactly one index, which we denote by $r$ for the rest of the proof. We thus have $\mathbf v=x_1\cdots x_{r-1}x_{r+1}\cdots x_n$.

Consider two subcases.

\smallskip

\textbf{Case 2.1}: $\mathbf a(\mathbf v_{r,g}\varphi)\mathbf b=z_0x_1z_1\cdots x_{2N+1}z_{2N+1}$ for some $g\in\{1,\dots,r-1,r+1,\dots,n\}$.

Under the premise of Case~2, this is only possible if the words $\mathbf a$ and $\mathbf b$ are empty. Then
\[
\mathbf v\varphi=(x_1\cdots x_{r-1}x_{r+1}\cdots x_n)\varphi\in\{x_1\cdots x_{2N+1},\,x_1\cdots x_N\cdot x_{N+2}\cdots x_{2N+1}\}.
\]
Since $N=m+1$ and $\mathbf v$ contains fewer than $m$ variables, there is an index $h\ne r$, $1\le h\le n$, such that the word $x_h\varphi$ contains a factor of the form $x_dx_{d+1}$. Hence the word $\mathbf v_{r,h}\varphi$ cannot equal $z_0x_1z_1\cdots x_{2N+1}z_{2N+1}$, so $\mathbf v_{r,h}\varphi\in\{x_1\cdots x_{2N+1},x_1\cdots x_N\cdot x_{N+2}\cdots x_{2N+1}\}$. Since $\mathbf v\varphi$ belongs to the same set and $\alf(x_r\varphi)\cap\alf(\mathbf v\varphi)=\varnothing$, this implies that
\[
x_r\varphi=x_{N+1},\ \mathbf v_{r,h}\varphi=x_1\cdots x_{2N+1},\ \text{ and }\ \mathbf v\varphi= x_1\cdots x_N\cdot x_{N+2}\cdots x_{2N+1}.
\]
If $x_{N+2}\in\alf(x_\ell\varphi)$ for some $\ell<r$, then the variable $x_{N+2}$ precedes the variable $x_{N+1}$ in the linear word $\mathbf v_{\ell,r}\varphi$, contradicting~\eqref{eq:subst}. Therefore, $(x_1\cdots x_{r-1})\varphi$ is a prefix of $x_1\cdots x_N$. By a similar argument, we can show that $(x_{r+1}\cdots x_n)\varphi$ is a suffix of $x_{N+2}\cdots x_{2N+1}$. Since $x_r\varphi=x_{N+1}$ and $\mathbf v\varphi=(x_1\cdots x_{r-1}x_{r+1}\cdots x_n)\varphi= x_1\cdots x_N\cdot x_{N+2}\cdots x_{2N+1}$, this implies that $(x_1\cdots x_{r-1})\varphi=x_1\cdots x_N$ and $(x_{r+1}\cdots x_n)\varphi=x_{N+2}\cdots x_{2N+1}$. Hence,
\[
(x_1\cdots x_n)\varphi=(x_1\cdots x_{r-1})\varphi\cdot x_r\varphi\cdot(x_{r+1}\cdots x_n)\varphi=x_1\cdots x_{2N+1}=\mathbf v_{r,h}\varphi.
\]

\smallskip

\textbf{Case 2.2}: $\mathbf a(\mathbf v_{r,j}\varphi)\mathbf b\ne z_0x_1z_1\cdots x_{2N+1}z_{2N+1}$ for all $j\in\{1,\dots,r-1,r+1,\dots,n\}$.

Fix any $j\in\{1,\dots,n\}\setminus\{r\}$. We have $\mathbf a(\mathbf v_{r,j}\varphi)\mathbf b\in\{x_1\cdots x_{2N+1},x_1\cdots x_N\cdot x_{N+2}\cdots x_{2N+1}\}$.  Since $\mathbf a(\mathbf v\varphi)\mathbf b$ belongs to the same set and $\alf(x_r\varphi)\cap\alf(\mathbf v\varphi)=\varnothing$, this implies that
\[
x_r\varphi=x_{N+1},\  \mathbf a(\mathbf v_{r,j}\varphi)\mathbf b=x_1\cdots x_{2N+1}\ \text{ and }\ \mathbf a(\mathbf v\varphi)\mathbf b= x_1\cdots x_N\cdot x_{N+2}\cdots x_{2N+1}.
\]
This is only possible if neither $\mathbf a$ nor $\mathbf b$ contains $x_{N}x_{N+2}$ as a factor. If $x_{N+2}\in\alf(x_\ell\varphi)$ for some $\ell<r$, then the variable $x_{N+2}$ precedes the variable $x_{N+1}$ in the linear word $\mathbf a(\mathbf v_{\ell,r}\varphi)\mathbf b$, contradicting~\eqref{eq:subst}. Therefore, $\mathbf a((x_1\cdots x_{r-1})\varphi)$ is a prefix of $x_1\cdots x_N$. By a similar argument, we can show that $(x_{r+1}\cdots x_n)\varphi\mathbf b$ is a suffix of $x_{N+2}\cdots x_{2N+1}$. Since $x_r\varphi=x_{N+1}$ and $\mathbf a(\mathbf v\varphi)\mathbf b= x_1\cdots x_N\cdot x_{N+2}\cdots x_{2N+1}$, this implies that $\mathbf a((x_1\cdots x_{r-1})\varphi)=x_1\cdots x_N$ and $(x_{r+1}\cdots x_n)\varphi\mathbf b=x_{N+2}\cdots x_{2N+1}$. Hence,
\[
\mathbf a((x_1\cdots x_n)\varphi)\mathbf b=\mathbf a((x_1\cdots x_{r-1})\varphi)\cdot x_r\varphi\cdot(x_{r+1}\cdots x_n)\varphi\mathbf b=x_1\cdots x_{2N+1}= \mathbf a(\mathbf v_{r,j}\varphi)\mathbf b.
\]
This implies $(x_1\cdots x_n)\varphi=\mathbf v_{r,j}\varphi$, and we are done.
\end{proof}

\begin{theorem}
\label{thm:nfb-CkxCk-}
For any $n,m\ge 3$, the identities common to the semirings $\mC_n$ and $\mC^-_m$ do not admit a finite basis.
\end{theorem}

\begin{proof}
The common identities of the semirings $\mC_n$ and $\mC^-_m$ are precisely the identities of their direct product $\mC_n\times\mC^-_m$. We will show that Proposition~\ref{prop:nfb-general} applies to the variety $\mathfrak{C}_{n,m}:=\var(\mC_n\times\mC^-_m)$. Since $n,m\ge 3$, the semirings $\mC_3$ and $\mC^-_3$ form subsemirings in $\mC_n$ and $\mC^-_m$, respectively. Hence, both $\mC_3$ and $\mC^-_3$ belong to the variety $\mathfrak{C}_{n,m}$. It remains to show that $\mathfrak{C}_{n,m}$ is contained in the variety $\mathfrak{T}$, which amounts to verifying that, for each $k\ge2$, the identity $\tau_k$ holds in both $\mC_n$ and $\mC^-_m$.

The word $x_1\cdots x_k\cdot x_{k+2}\cdots x_{2k+1}$ occurs as a subword in $x_1\cdots x_{2k+1}$. By Corollary~\ref{cor:subwordinC-}, every semiring $\mC^-_m$ satisfies the identity
\[
x_1\cdots x_{2k+1}\preccurlyeq x_1\cdots x_k\cdot x_{k+2}\cdots x_{2k+1},
\]
which clearly implies $\tau_k$.

Similarly, since the word $x_1\cdots x_{2k+1}$ occurs as a subword in $z_0x_1z_1x_2\cdots x_{2k+1}z_{2k+1}$, by Lemma~\ref{lem:subword}, every semiring $\mC_n$ satisfies the identity
\[
x_1x_2\cdots x_{2k+1}\preccurlyeq z_0x_1z_1x_2\cdots x_{2k+1}z_{2k+1},
\]
which also implies $\tau_k$.
\end{proof}

For $n=m=3$, Theorem~\ref{thm:nfb-CkxCk-} implies that the direct product of the \fb{} 5-element semirings $\mC_3$ and $\mC^-_3$  is \nfb. Until recently, examples of \fb{} finite \ais{} with a \nfb{} direct product were unknown. An alternative example was found by Yue, Ren, and Gao~\cite[Proposition 3.4 + Theorem 3.10]{yrg26}. The \ais{}s in their example are denoted $D_2$ and $S_{53}$; the first is isomorphic to the semiring $\mC^-_2$, while the other is isomorphic to the subsemiring of the semiring $\mC_3$ formed by the transformations 123, 233, 333 (here $i_1i_2i_3$ denotes the map $j\mapsto i_j$, $j=1,2,3$). This provokes the question whether the direct product $\mC_3\times\mC^-_2$ is \fb. The `dual' direct product $\mC_2\times\mC^-_3$ is \fb; moreover, we have the following consequence of a general result from universal algebra.

\begin{proposition}
\label{prop:regularization}
For any $m$, the identities common to the semirings $\mC_2$ and $\mC^-_m$ admit a finite basis.
\end{proposition}

\begin{proof}
In universal algebra, an identity $\mathbf U\approx\mathbf V$ is called \emph{regular} if $\alf(\mathbf U)=\alf(\mathbf V)$. A variety is called \emph{strongly irregular} if it satisfies an identity $x\approx\mathbf W$ such that some variable $y\ne x$ occurs in $\mathbf W$. The result we need states that, for algebras with finitely many operations, the set of all regular identities of any \fb{} strongly irregular variety has a finite basis. This was proved by Mel'nik~\cite{Mel71}.\footnote{Since~\cite{Mel71} may be difficult to access, the reader may instead consult~\cite{Rom86}, where the result is established as Corollary~7.7. We follow the terminology of the latter article; Mel'nik called regular identities \emph{normal} and strongly irregular varieties \emph{perfect}.}

It is known from~\cite[Lemma 1.1]{ShaoRen15} (and also follows from Proposition~\ref{prop:idCn}) that $\mathsf{Id}(\mC_2)$ consists of all regular identities. Hence, the identities common to the semirings $\mC_2$ and $\mC^-_m$ are precisely the regular identities of the latter semiring. Since $\mC^-_m$ satisfies the identity $x\approx x+xy$ (equivalent to~\eqref{eq:xy<x}), the variety $\var\mC^-_m$ is strongly irregular. Since this variety is \fb{} by Theorem~\ref{thm:Cn-}, Mel'nik's result applies.
\end{proof}

\section{Pseudovarieties and varieties generated by Catalan semirings}
\label{sec:varieties}

\subsection{\texorpdfstring{Pseudovariety and variety generated by the semirings $\mC_n$, $n\ge2$}{Pseudovariety and variety generated by the semirings Cn}}
\label{subsec:pseudoCn}
A class $\mathfrak{P}$ of finite algebras is called a \emph{pseudovariety} if
\begin{itemize}
  \item for any algebra $\mA\in\mathfrak{P}$, all its homomorphic images and subalgebras belong to $\mathfrak{P}$,
  \item for any finite family of algebras $\mA_1,\dots,\mA_k\in\mathfrak{P}$, their direct product $\mA_1\times\dots\times\mA_k$ belongs to $\mathfrak{P}$.
\end{itemize}
In particular, the class of all finite algebras in any variety constitutes a pseudovariety. We say that a pseudovariety $\mathfrak{P}$ is \emph{defined by a set $\Sigma$ of identities} if $\mathfrak{P}$ is the collection of all finite algebras in the variety defined by $\Sigma$. Such pseudovarieties are called \emph{equational}.

Given a class $\mathfrak{A}$ of finite algebras, the least pseudovariety containing $\mathfrak{A}$ is called the \emph{pseudovariety generated by $\mathfrak{A}$} and is denoted by $\pvar\mathfrak{A}$. We write $\pvar\mA$ rather than $\pvar\{\mA\}$ for a class consisting of a single finite algebra $\mA$. 
When $\mathfrak{A}$ consists of finitely many algebras, Eilenberg and Sch\"utzenberger~\cite[Corollary 4]{ES76} proved that $\pvar\mathfrak{A}$ is equational and defined by the set of all identities satisfied by every algebra in $\mathfrak{A}$. For infinite classes, this need not be true: for instance, the pseudovariety $\pvar\{C_n\}_{n\ge 2}$ generated by all Catalan monoids is not equational and is much smaller than the pseudovariety defined by the common identities of these monoids. Indeed, \cite[Theorem~4.1.10]{Pin86} implies that $\pvar\{C_n\}_{n\ge 2}$ is the class of all finite semigroups in which every principal two-sided ideal has a unique generator (so-called  $\mathrsfs{J}$-\emph{triv\-i\-al} semigroups). On the other hand, it is known (and easy to see) that the only identities satisfied by every Catalan monoid are consequences of the associative law and thus hold in every semigroup; hence, the pseudovariety defined by these identities is the class of all finite semigroups.

As mentioned in Subsection~\ref{subsec:catalanmonoids}, monoid pseudovarieties are the algebraic counterparts of so-called varieties of recognizable languages within the framework of Eilenberg's correspondence. Pol\'ak~\cite{Polak04a} established a broader correspondence, using \ais{} pseudovarieties on the algebraic side, which allowed him to relax Eilenberg's conditions on language classes. In~\cite{KP10}, Kl\'ima and Pol\'ak essentially studied the recognizable languages related to the pseudovarieties $\pvar\mC_n$, $n=2,3,\dotsc$, under Pol\'ak's correspondence.\footnote{Recall that \cite{KP10} works with the semiring of all staircase triangular Boolean $n\times n$ matrices, which is isomorphic to $\mC_n$, rather than $\mC_n$ itself.} The following characterization of the pseudovariety $\pvar\{\mC_n\}_{n\ge 2}$ could, in principle, have been extracted from the language-theoretic considerations in~\cite{KP10}, but here we give a direct and purely algebraic proof.

\begin{theorem}
\label{thm:pvarCn}
The pseudovariety generated by the semirings $\mC_n$ for $n=2,3,\dotsc$ is defined by identities \eqref{eq:xy>x} and \eqref{eq:yx>x}.
\end{theorem}

Our proof of Theorem~\ref{thm:pvarCn} relies on a property of endomorphism semirings of semilattices. Given a semilattice $\mQ=(Q;\,+)$, we equip the set $\End(\mQ)$ of all its endomorphisms with addition and multiplication as follows. For all endomorphisms $\alpha,\beta\colon Q\to Q$, their sum $\alpha+\beta$ is defined pointwise by letting for each $q\in Q$,
\[
q(\alpha+\beta):=q\alpha+q\beta.
\]
The product $\alpha\beta$ is just the composition of $\alpha$ and $\beta$ as maps; that is, for each $q\in Q$,
\[
q(\alpha\beta):=(q\alpha)\beta.
\]
It is easy to verify that $(\End(\mQ);\,+,\cdot)$ is an \ais.

Recall the +-order $\leqslant$ on the semilattice $\mQ=(Q;\,+)$ defined by
\[
p\leqslant q\ \text{ if and only if }\ p+q=q.
\]
The \emph{height} of $\mQ$ is the maximum size of a chain in the partially ordered set $(Q;\leqslant)$, provided that such a maximum exists.

An endomorphism $\alpha\in\End(\mQ)$ is \emph{extensive} if $q\leqslant q\alpha$ for all $q\in Q$. The following result is a semiring analogue of a fact established for semigroups in~\cite[Proposition~11]{Vo04} and is proved in essentially the same way. Yet we provide a full proof for completeness.

\begin{proposition}
\label{prop:increasendo}
Let $\mQ=(Q;\,+)$ be a semilattice of height $n$. Every subsemiring $\mathcal{S}=(S;\,+,\cdot)$ of\/ $\End(\mQ)$ consisting of extensive endomorphisms satisfies all identities of the semiring $\mC_n$.
\end{proposition}

\begin{proof}
By Proposition~\ref{prop:idCn}, an identity $\mathbf U\approx\mathbf V$ holds in $\mC_n$ if and only if the polynomials $\mathbf U$ and $\mathbf V$ have the same set of subwords of length less than $n$. Take an arbitrary substitution $\varphi\colon\alf(\mathbf{UV})\to S$. We have to show that $\mathbf{U}\varphi=\mathbf{V}\varphi$, that is, $q(\mathbf{U}\varphi)=q(\mathbf{V}\varphi)$ for all $q\in Q$.

Fix an arbitrary $q\in Q$. By symmetry, it suffices to verify that $q(\mathbf{U}\varphi)\leqslant q(\mathbf{V}\varphi)$.

Let $\mathbf{U}=\mathbf{u}_1+\dots+\mathbf{u}_\ell$. If $q(\mathbf u_i\varphi)=q$ for all $i=1,\dots,\ell$, then $q(\mathbf{U}\varphi)=q\leqslant q(\mathbf{V}\varphi)$ because the endomorphism $\mathbf{V}\varphi$ is extensive. Suppose now that the set
\[
I:=\{i\mid 1\le i\le \ell,\ q(\mathbf u_i\varphi)>q\}
\]
is nonempty. For each $i\in I$, denote by $\mathbf u_{i,1}$ the longest prefix of the word $\mathbf u_i$ such that $q(\mathbf u_{i,1}\varphi)=q$, and let $x_{i,1}$ be the variable that follows $\mathbf u_{i,1}$ in $\mathbf u_i$, so that $\mathbf u_i=\mathbf u_{i,1}x_{i,1}\mathbf w_{i,1}$ for some (possibly empty) word $\mathbf w_{i,1}$. Then
\begin{equation}
\label{eq:step1}
q_{i,1}:=q((\mathbf u_{i,1}x_{i,1})\varphi)=(q(\mathbf u_{i,1}\varphi))(x_{i,1}\varphi)=q(x_{i,1}\varphi)\geqslant q
\end{equation}
because the endomorphism $x_{i,1}\varphi$ is extensive, and by the choice of the prefix $\mathbf u_{i,1}$, the inequality $q_{i,1}\geqslant q$ is in fact strict. Now denote by $\mathbf u_{i,2}$ the longest prefix of the word $\mathbf w_{i,1}$ such that $q_{i,1}(\mathbf u_{i,2}\varphi)=q_{i,1}$, and let $x_{i,2}$ be the variable that follows $\mathbf u_{i,2}$ in $\mathbf w_{i,1}$, so that $\mathbf u_i=\mathbf u_{i,1}x_{i,1}\mathbf u_{i,2}x_{i,2}\mathbf w_{i,2}$ for some (possibly empty) word $\mathbf w_{i,2}$. Then
\begin{equation}
\label{eq:step2}
q_{i,2}:=q_{i,1}((\mathbf u_{i,2}x_{i,2})\varphi)=(q_{i,1}(\mathbf u_{i,2}\varphi)(x_{i,2}\varphi)=q_{i,1}(x_{i,2}\varphi)>q_{i,1},
\end{equation}
and substituting the expression for $q_{i,1}$ from~\eqref{eq:step1} in the expression for $q_{i,2}$ in~\eqref{eq:step2} yields
\[
q_{i,2}=q((\mathbf u_{i,1}x_{i,1}\mathbf u_{i,2}x_{i,2})\varphi)=q((x_{i,1}x_{i,2})\varphi).
\]
Continuing this process, we finally arrive at the decomposition
\begin{equation}
\mathbf u_i=\mathbf u_{i,1}x_{i,1}\mathbf u_{i,2}x_{i,2}\cdots x_{i,m_i}\mathbf u_{i,m_i+1}
\label{eq:decomposed}
\end{equation}
such that $q(\mathbf u_i\varphi)=q((x_{i,1}\cdots x_{i,m_i})\varphi)$ and
\[
q_{i,m_i}>q_{i,m_i-1}>\dots>q_{i,1}>q
\]
where $q_{i,j}:=q((x_{i,1}\cdots x_{i,j})\varphi)$ for $j=1,2,\dots,m_i$. Since the maximum size of a chain in $(Q;\leqslant)$ is $n$, we conclude that $m_i<n$.

Since the polynomials $\mathbf U$ and $\mathbf V$ have the same set of subwords of length less than $n$, the word $x_{i,1}\cdots x_{i,m_i}$, which by~\eqref{eq:decomposed} is a subword of length less than $n$ in the summand $\mathbf u_i$ of $\mathbf U$, must occur as a subword in a summand of $\mathbf V$; choose such a summand and denote it by $\mathbf v_i$. (It is possible that the same summand of $\mathbf V$ is chosen for different summands $\mathbf u_i$, but this is not an issue due to idempotency of addition.) Then
\[
\mathbf v_i=\mathbf v_{i,1}x_{i,1}\mathbf v_{i,2}x_{i,2}\cdots x_{i,m_i}\mathbf v_{i,m_i+1}
\]
for some (possibly empty) words $\mathbf v_{i,1},\dots,\mathbf v_{i,m_i+1}$. Since the endomorphisms in $\mathcal{S}$ are extensive and all endomorphisms preserve the +-order, we obtain that
\begin{equation}\label{eq:qu<qv}
q(\mathbf v_i\varphi)\geqslant q((x_1x_2\cdots x_m)\varphi)=q(\mathbf u_i\varphi)
\end{equation}
for each $i\in I$. It follows that
\begin{align*}
q(\mathbf{U}\varphi) &= q\bigl((\mathbf u_1+\cdots+\mathbf u_\ell)\varphi\bigr) \\
                     &= q\Bigl(\bigl(\sum_{i\in I}\mathbf u_{i}\bigr)\varphi\Bigr) && \text{by the choice of $I$} \\
                     &\leqslant q\Bigl(\bigl(\sum_{i\in I}\mathbf v_{i}\bigr)\varphi\Bigr) && \text{by \eqref{eq:qu<qv}}\\
                     &\leqslant q(\mathbf{V}\varphi) && \text{since $\sum_{i\in I}\mathbf v_i\preccurlyeq\mathbf V$}
\end{align*}
as required.
\end{proof}

We record two corollaries of Proposition~\ref{prop:increasendo}. The first (taken from the first-named author's preprint~\cite{Gusev23}) is a semiring analogue of the equivalence established in~\cite{Vo04} for the semigroup case.

A Boolean matrix is called \emph{reflexive} if all entries on its main diagonal are 1s. (Under the standard representation of binary relations as Boolean matrices, reflexive matrices correspond to reflexive relations, hence the name.) The set of all reflexive Boolean $n\times n$ matrices is closed under matrix addition and multiplication so it forms a subsemiring of the \ais{} of all Boolean $n\times n$ matrices. We denote by $R_n$ and $\mR_n$ the monoid and, respectively, the \ais{} of all reflexive Boolean $n\times n$ matrices. The upper triangular matrices from $R_n$ are called \emph{unitriangular}; we denote by $UT_n$ and $\mathcal{UT}_n$ the monoid and, respectively, the \ais{} of all unitriangular Boolean $n\times n$ matrices.

Observe that the semiring $\mR_n$ is much larger in size than its subsemiring $\mathcal{UT}_n$, which in turn is exponentially larger in size than the Catalan semiring $\mC_n$ for $n>2$:
\[
|R_n|=2^{n^2-n} > |UT_n|=\sqrt{|R_n|}=2^{\tfrac{n^2-n}2}> |C_n|=\frac1{n+1}\binom{2n}n\sim \frac{2^{2n}}{n^{3/2}\sqrt{\pi}}.
\]
This makes the following fact somewhat unexpected.
\begin{corollary}
\label{cor:Cn=UTn=Rn}
For every $n\ge 2$, the semirings $\mC_n$, $\mathcal{UT}_n$, and $\mR_n$ satisfy the same identities; in particular, the identities \eqref{eq:xy>x}, \eqref{eq:yx>x}, and $\kappa_n$ constitute an \ib{} for each of these semirings.
\end{corollary}

\begin{proof}
As mentioned in Subsection~\ref{subsec:klimapolak}, the semiring $\mC_n$ is isomorphic to the semiring of all staircase triangular Boolean $n\times n$ matrices, which, by definition, is a subsemiring of the semiring $\mUT_n$. The latter is a subsemiring of $\mR_n$. Hence, every identity holding in $\mR_n$ holds in both $\mUT_n$ and $\mC_n$. It remains to verify that $\mR_n$ satisfies all identities of the semiring $\mC_n$.

Denote by $\mP[n]=(P[n];\,+)$ the semilattice of all nonempty subsets of $[n]$ with addition defined as set-theoretic union. Then the +-order on $P[n]$ is set-theoretic inclusion, and hence the height of $\mP[n]$ is $n$.

Each matrix $r=(r_{ij})\in R_n$ defines a map $\bar r\colon P[n]\to P[n]$ as follows: for a set $S\in P[n]$,
\[
S\bar r:=\{j\in n\mid i\in S\ \&\ r_{ij}=1\}.
\]
It is easy to see that $\bar r$ is an endomorphism of the semilattice $\mP[n]$; it is extensive since $r_{ii}=1$ for all $i\in [n]$, and so $S\bar r\supseteq S$ for all $S\in P[n]$. The map $\mR_n\to\End(\mP[n])$ defined by $r\mapsto\bar r$ is clearly a homomorphism. This homomorphism is injective since the sets $\{1\}\bar r$, \dots, $\{n\}\bar r$ uniquely determine the matrix $r$; namely, $r_{ij}=1$ if and only if $j\in\{i\}\bar r$. Now Proposition~\ref{prop:increasendo} applies, showing that  $\mR_n$ satisfies all identities of the semiring $\mC_n$.

That the identities \eqref{eq:xy>x}, \eqref{eq:yx>x}, and $\kappa_n$ constitute an \ib{} for each of the semirings $\mC_n$, $\mathcal{UT}_n$, and $\mR_n$ now follows from Theorem~\ref{thm:basisCn}.
\end{proof}

The second corollary of Proposition~\ref{prop:increasendo} is a key to the proof of Theorem~\ref{thm:pvarCn}.

\begin{corollary}
\label{cor:Cnfinite}
Every \ais{} $\mathcal{S}=(S;\,+,\cdot)$ satisfying \eqref{eq:xy>x} and \eqref{eq:yx>x}, and such that the semilattice $(S;\,+)$ has height $n$, belongs to the variety $\var\mC_{n+1}$.
\end{corollary}

\begin{proof}
Take a symbol $\bot\notin S$ and extend addition to the set $S^\bot:=S\cup\{\bot\}$ by letting $\bot+t=t+\bot:=t$ for all $t\in S^\bot$. It is easily seen that $(S^\bot;\,+)$ becomes a semilattice, which we denote by $\mS^\bot$, and that $\bot\leqslant t$ in the +-order on $\mS^\bot$ for all $t\in S^\bot$.

For every $s\in S$, define a map $\overline{s}\colon S^\bot\to S^\bot$ by setting, for all $t\in S^\bot$,
\[
t\overline{s}:=\begin{cases}
s&\text{if }\ t=\bot,\\
ts &\text{if }\ t\in S.
\end{cases}
\]
We claim that $\overline{s}$ is an endomorphism of the semilattice $\mS^\bot$. This amounts to verifying the equality $(t_1+t_2)\overline{s}=t_1\overline{s}+t_2\overline{s}$ for all $t_1,t_2\in S^\bot$. If $t_1,t_2\in S$, the equality holds since multiplication distributes over addition in the \ais{} $\mS$. If $t_1=t_2=\bot$, then $t_1+t_2=\bot$, and we have $(t_1+t_2)\overline{s}=s$ and $t_1\overline{s}+t_2\overline{s}=s+s=s$. By commutativity of addition, it remains to consider the case where $t_1\in S$ and $t_2=\bot$. Then $(t_1+\bot)\overline{s}=t_1s$ and $t_1\overline{s}+\bot\overline{s}=t_1s+s$, but $t_1s=t_1s+s$ in $\mS$, since this \ais{} satisfies identity \eqref{eq:yx>x} equivalent to $xy\approx xy+y$.

Every endomorphism $\overline{s}$ is extensive. Indeed, $\bot\leqslant\bot\overline{s}$, since, as mentioned, $\bot$ is the least element with respect to the +-order on $\mS^\bot$, and $t\leqslant ts$ for every $t\in S$, since $\mS$ satisfies \eqref{eq:xy>x}. It is routine to verify that the map $\mS\to\End(\mS^\bot)$ defined by $s\mapsto\overline{s}$ is a semiring homomorphism. Since $\overline{s_1}=\overline{s_2}$ implies $s_1=\bot\overline{s_1}=\bot\overline{s_2}=s_2$ for any $s_1,s_2\in S$, this homomorphism is injective. We conclude that $\mS$ is isomorphic to a subsemiring of $\End(\mS^\bot)$ consisting of extensive endomorphisms.

The height of the semilattice~$\mS^\bot$ is $n+1$, since the longest chains in the partially ordered set $(S^\bot;\leqslant)$ are those of $(S;\leqslant)$ with $\bot$ prepended. By Proposition~\ref{prop:increasendo}, the semiring $\mS$ satisfies all identities of the Catalan semiring $\mC_{n+1}$. Thus, $\mS\in\var\mC_{n+1}$.
\end{proof}

\begin{proof}[Proof of Theorem~\ref{thm:pvarCn}]
We must verify that a finite \ais{} $\mS=(S;\,+,\cdot)$ belongs to the pseudovariety $\pvar\{\mC_n\}_{n\ge 2}$ if and only if $\mS$ lies in the variety $\mathfrak{C}$ defined by \eqref{eq:xy>x} and \eqref{eq:yx>x}.

The `only if' direction is immediate: all semirings $\mC_n$, $n=2,3,\dotsc$, satisfy \eqref{eq:xy>x} and \eqref{eq:yx>x}, and hence $\pvar\{\mC_n\}_{n\ge 2}\subset\mathfrak{C}$.

For the `if' direction, Corollary~\ref{cor:Cnfinite} ensures that $\mS\in\var\mC_{n+1}$, where $n$ is the height of the finite semilattice $(S;\,+)$. By \cite[Corollary 4]{ES76}, the finite members of the variety $\var\mC_{n+1}$ belong to the pseudovariety $\pvar\mC_{n+1}$.
\end{proof}

Does the analogue of Theorem~\ref{thm:pvarCn} hold for varieties; that is, does the variety $\var\{\mC_n\}_{n\ge 2}$ coincide with the variety $\mathfrak{C}$ defined by \eqref{eq:xy>x} and \eqref{eq:yx>x}? Since all semirings $\mC_n$ ($n=2,3,\dotsc$) satisfy \eqref{eq:xy>x} and \eqref{eq:yx>x}, we have the inclusion $\var\{\mC_n\}_{n\ge 2}\subseteq\mathfrak{C}$, and Theorem~\ref{thm:pvarCn} implies that the two varieties contain exactly the same finite members. However, in general, a variety $\mathfrak{V}$ need not coincide with a subvariety that contains the same finite members as $\mathfrak{V}$. For instance, the variety of groups of exponent dividing 665 has a proper subvariety containing all finite groups of that exponent, as follows from a combination of the celebrated results of Adian~\cite{Adian75} and Zelmanov~\cite{Zel90}.

We will address the question of whether the equality $\var\{\mC_n\}_{n\ge 2}=\mathfrak{C}$ holds by analyzing a model of the $\mathfrak{C}$-free semiring on a set. For an \ais{} variety $\mathfrak{V}$ and a set $X$, the $\mathfrak{V}$-\emph{free semiring} on $X$ is an \ais{} $\mF\in\mathfrak{V}$ satisfying the following two conditions (as illustrated in the diagram below):
\begin{itemize}
  \item[(F1)] there is an injective map $\iota\colon X\hookrightarrow\mF$ such that $X\iota$ generates $\mF$;
  \item[(F2)] for every \ais{} $\mathcal{S}=(S;\,+,\cdot)\in\mathfrak{V}$ and every substitution $\mu\colon X\to S$, there is a unique homomorphism $\varphi\colon\mF\to\mS$ such that $\mu=\iota\varphi$.
\end{itemize}
\begin{center}
\begin{tikzcd}
X \arrow[r, hook, "\iota"] \arrow[dr, "\mu"'] & \mF=\langle X\iota\rangle \arrow[d, dashed, "\varphi"] \\
& \mS
\end{tikzcd}
\end{center}

Let $X^+$ stand for the set of all words with variables from $X$. Consider the set $SC_f(X)$ of all finite nonempty subsets of $X^+$ that are closed under taking subwords, and endow it with addition + and multiplication $\circ$ by setting, for all $U,V\in SC_f(X)$,
\[
U+V:=U\cup V, \quad U\circ V:=U\cup V\cup UV,
\]
where $UV:=\{\mathbf{uv}\mid \mathbf u\in U,\ \mathbf v\in V\}$. It is easy to see that $U+V,U\circ V\in SC_f(X)$ for all $U,V\in SC_f(X)$ and that $\mathcal{SC}_f(X):=(SC_f(X);\,+,\circ)$ is an \ais.

\begin{lemma}
\label{lem:freeCn}
The semiring $\mathcal{SC}_f$ is the $\mathfrak{C}$-free semiring on $X$.
\end{lemma}

\begin{proof}
First, note that the semiring $\mathcal{SC}_f(X)$ belongs to the variety $\mathfrak{C}$. Indeed, we have $U\subset U\circ V$ and $U\subset V\circ U$ for all $U,V\in SC_f(X)$, and hence the identities $x\preccurlyeq xy$ and $x\preccurlyeq yx$, which define $\mathfrak{C}$, hold in $\mathcal{SC}_f(X)$.

The map $\iota\colon X\to SC_f(X)$ defined by $x\iota:=\{x\}$ for all $x\in X$ is injective. Every element $W\in SC_f(X)$ can be expressed as $\sum_{\mathbf{w}\in W}\mathbf{Sub}(\mathbf{w})$, where $\mathbf{Sub}(\mathbf{w})$ denotes the set of all subwords of $\mathbf{w}$ (including $\mathbf{w}$ itself). Furthermore, for any word $\mathbf{w}=x_1\cdots x_k$, a straightforward induction on $k$ shows that $\mathbf{Sub}(\mathbf{w})=\{x_1\}\circ\cdots\circ\{x_k\}$. Hence, $X\iota$ generates the semiring $\mathcal{SC}_f(X)$, and so condition (F1) holds.

Take an arbitrary \ais{} $\mathcal{S}=(S;\,+,\cdot)\in\mathfrak{C}$ and an arbitrary substitution $\mu\colon X\to S$. Extend $\mu$ (keeping the same notation) to a map $\mu\colon X^+\to S$ by setting $\mathbf w\mu:=x_1\mu\cdots x_k\mu$ for every word $\mathbf w=x_1\cdots x_k$. Now, let $W\varphi:=\sum_{\mathbf w\in W}\mathbf w\mu$ for all $W\in SC_f(X)$. By construction, $x\mu=\{x\}\varphi=(x\iota)\varphi$ for all $x\in X$.

To complete the verification of condition (F2), it remains to show that $\varphi$ is a homomorphism, since the uniqueness of $\varphi$ then follows from the fact that its value on each generator $x\iota$ is uniquely determined to be $x\mu$.

By construction, it is immediate that $(U+V)\varphi=U\varphi+V\varphi$ for all $U,V\in SC_f(X)$.

Denote the polynomial $\displaystyle{\sum\limits_{\mathbf u\in U,\ \mathbf v\in V}\mathbf{uv}}$ by $\mathbf W$. Then
\[
(U\varphi)\cdot(V\varphi)=\sum_{\mathbf u\in U}\mathbf u\mu\cdot\sum_{\mathbf v\in V}\mathbf v\mu=\sum_{\mathbf u\in U,\ \mathbf v\in V}(\mathbf{uv})\mu=\mathbf W\mu.
\]
It is easy to see that $U\circ V$ consists of all subwords of $\mathbf W$, that is, $U\circ V=\mathbf{Sub(W)}$, in the notation from the proof of Theorem~\ref{thm:basisCn}. As shown in that proof, identities \eqref{eq:xy>x} and \eqref{eq:yx>x} imply the identity $\mathbf{W}\approx\mathbf{Sub(W)}$. Since the semiring $\mS$ satisfies these identities, we obtain
\[
(U\varphi)\cdot(V\varphi)=\mathbf{W}\mu=\mathbf{Sub(W)}\mu=(U\circ V)\varphi
\]
for all $U,V\in SC_f(X)$. Thus, $\varphi$ is a homomorphism.
\end{proof}

Now we are in a position to prove the equality $\var\{\mC_n\}_{n\ge 2}=\mathfrak{C}$.

\begin{theorem}
\label{thm:varCn}
The variety generated by the semirings $\mC_n$ for $n=2,3,\dotsc$ is defined by identities \eqref{eq:xy>x} and \eqref{eq:yx>x}.
\end{theorem}

\begin{proof}
As mentioned, $\var\{\mC_n\}_{n\ge 2}\subseteq\mathfrak{C}$. The reverse inclusion $\var\{\mC_n\}_{n\ge 2}\supseteq\mathfrak{C}$ will be established once we prove that the $\mathfrak{C}$-free semiring on an infinite set is residually finite. (Recall that an algebra $\mA$ is called \emph{residually finite} if every two distinct elements of $\mA$ have distinct images under some homomorphism from $\mA$ onto a finite algebra.) Indeed, by a standard fact of universal algebra \cite[Exercise II.8.11]{BuSa81}, every residually finite algebra embeds into the direct product of its finite homomorphic images and hence belongs to the variety generated by these images. By Theorem~\ref{thm:pvarCn}, finite semirings from $\mathfrak{C}$ belong to $\var\{\mC_n\}_{n\ge 2}$, and therefore so do residually finite semirings from $\mathfrak{C}$. It remains to invoke another standard fact of universal algebra \cite[Exercise II.11.2]{BuSa81}: every variety is generated by its free algebra on any infinite set.

In view of Lemma~\ref{lem:freeCn}, we need to show that the semiring $\mathcal{SC}_f(X)$, with infinite $X$, is residually finite. Take two arbitrary distinct elements $W_1,W_2\in SC_f(X)$. Since $W_1$ and $W_2$ are finite, there exist a finite subset $Y\subset X$ and an integer $k$ such that $\alf(\mathbf w)\subseteq Y$ and $|\mathbf{w}|<k$ for all words $\mathbf w\in W_1\cup W_2$. Let $Y^{<k}$ denote the set of all words of length less than $k$ with variables in $Y$. Define a binary relation $\rho$ on the set $SC_f(X)$ by setting, for all $U,V\in SC_f(X)$,
\[
U\mathrel{\rho}V \iff U\cap Y^{<k}=V\cap Y^{<k}.
\]
It is routine to verify that $\rho$ is a congruence on the semiring $\mathcal{SC}_f(X)$. With $Y$ and $k$ as chosen, $W_1=W_1\cap Y^{<k}\ne W_2\cap Y^{<k}=W_2$. Hence $W_1$ and $W_2$ have distinct images under the natural homomorphism from $\mathcal{SC}_f(X)$ onto the quotient semiring $\mathcal{SC}_f(X)/\rho$. Since every $\rho$-class is uniquely determined by the common intersection of its elements with the finite set $Y^{<k}$, the semiring $\mathcal{SC}_f(X)/\rho$ is finite.
\end{proof}

\subsection{\texorpdfstring{Pseudovariety and variety generated by the semirings $\mC^-_n$, $n\ge2$}{Pseudovariety and variety generated by the semirings Cn}}
\label{subsec:pseudoCn-}
Now we turn to the semirings $\mC^-_n$. By Lemma~\ref{lem:xle1}, they satisfy identities~\eqref{eq:xy<x} and~\eqref{eq:yx<x}, which are equivalent to the \emph{absorption laws} $x+xy\approx x$ and $x+yx\approx x$, respectively. The additively idempotent semirings satisfying the absorption laws have been extensively studied under the name \emph{inclines}; see the monograph~\cite{CKR84}, the survey~\cite{KR04}, and the references therein. A classical example of an incline is the incline of all two-sided ideals of a ring $(R;\,+,\cdot)$, with the sum of two ideals $I$ and $J$ defined by $I+J:=\{i+j\mid i\in I,\ j\in J\}$ and their product defined by $I\cdot J:=\{i_1j_1+\cdots+i_kj_k\mid i_1\dots,i_k\in I,\ j_1\dots,j_k\in J\}$.

We prove that the family $\{\mC_n\}_{n\ge 2}$ generates the (pseudo)variety of all (finite) inclines.

\begin{theorem}
\label{thm:pvarCn-}
The pseudovariety generated by the semirings $\mC^-_n$ for $n=2,3,\dotsc$ is defined by identities \eqref{eq:xy<x} and \eqref{eq:yx<x}.
\end{theorem}

The proof plan is similar to that of Theorem~\ref{thm:pvarCn}, but there are technical subtleties arising from the structural differences between the identities satisfied by $\mC_n$ and $\mC^-_n$.

An endomorphism $\alpha$ of a semilattice $(Q;\,+)$ is \emph{decreasing} if $q\geqslant q\alpha$ for all $q\in Q$. The statement of the following result appears dual to that of Proposition~\ref{prop:increasendo}, but this duality is not exact, and the proofs differ substantially.

\begin{proposition}
\label{prop:decreasendo}
Let $\mQ=(Q;\,+)$ be a semilattice of height $n$. Every subsemiring $\mathcal{S}=(S;\,+,\cdot)$ of\/ $\End(\mQ)$ consisting of decreasing endomorphisms satisfies all identities of the semiring $\mC^-_{n+1}$.
\end{proposition}

\begin{proof}
By Lemma~\ref{lem:linearization}, it suffices to show that $\mS$ satisfies all identities of $\mC^-_n$ with a linear lower side. Let $x_1\cdots x_k\preccurlyeq\mathbf V$ be such an identity and let $\varphi\colon\{x_1,\dots,x_k\}\cup\alf(\mathbf{V})\to S$ be an arbitrary substitution. We have to show that $(x_1\cdots x_k)\varphi\leqslant\mathbf{V}\varphi$ in $\mS$, which means that $q((x_1\cdots x_k)\varphi)\leqslant q(\mathbf{V}\varphi)$ for all $q\in Q$.

For an arbitrary $q\in Q$, define $q_0:=q$ and $q_i:=q_{i-1}(x_i\varphi)$ for $i=1,\dots,k$. Since each endomorphism $x_i\varphi$ is decreasing, we have
\[
q_0\geqslant q_1\geqslant\cdots\geqslant q_k.
\]
At most $n-1$ of these inequalities can be strict, since the maximum size of a chain in $(Q;\leqslant)$ is $n$. Let $m:=\min\{n-1,k\}$ and fix an $m$-element subset $I\subseteq [k]$ containing all indices $i$ such that $q_{i-1}>q_i$. Let $1\le i_1<\cdots<i_m\le k$ be the elements of $I$; it is convenient to use also $i_0:=0$ and $i_{m+1}:=k+1$. With this notation, our choice of the set $I$ implies that $q_s(x_j\varphi)=q_s$ for each $s=0,1,\dots,m$ and each $i_s<j<i_{s+1}$, and hence the above chain of inequalities can be rewritten as
\[
q_0=\cdots=q_{i_1-1}\geqslant q_{i_1}=\cdots=q_{i_1-1}\geqslant q_{i_2}=\cdots=q_{i_m-1}\geqslant q_{i_m}=q_k.
\]
Since the identity $x_1\cdots x_k\preccurlyeq\mathbf V$ holds in the semiring $\mC^-_{n+1}$ and $m<n$, Proposition~\ref{prop:idCn-} implies that a summand $\mathbf v$ of $\mathbf V$ can be decomposed as
\[
\mathbf v=\mathbf v_1x_{i_1}^{e_1}\mathbf v_2\cdots \mathbf v_mx_{i_m}^{e_m}\mathbf v_{m+1},
\]
where $e_1,\dots,e_m\in\{0,1\}$ and the (possibly empty) words $\mathbf{v}_s$, $s=1,\dots,m+1$, are such that $\alf(\mathbf v_s)\subseteq\{x_{i_{s-1}+1},\dots, x_{i_s-1}\}$. Then we have $q_{i_{s-1}}(x_j\varphi)=q_{i_{s-1}}$ for all $x_j\in\alf(\mathbf v_s)$, and hence $q_{i_{s-1}}(\mathbf v_s\varphi)=q_{i_{s-1}}$. Using this, we obtain, for each $s=1,\dots,m$,
\begin{align*}
 q_{i_{s-1}}((\mathbf v_s\cdots \mathbf v_mx_{i_m}^{e_m}\mathbf v_{m+1})\varphi)&=(q_{i_{s-1}}(\mathbf v_s\varphi))((x_{i_s}^{e_s}\mathbf v_{s+1}\cdots \mathbf v_mx_{i_m}^{e_m}\mathbf v_{m+1})\varphi)\\
                   &=q_{i_{s-1}}((x_{i_s}^{e_s}\mathbf v_{s+1}\cdots \mathbf v_mx_{i_m}^{e_m}\mathbf v_{m+1})\varphi)\\
                   &\geqslant q_{i_s}((\mathbf v_{s+1}\cdots \mathbf v_mx_{i_m}^{e_m}\mathbf v_{m+1})\varphi).
\end{align*}
Here, the last inequality is justified as follows. If $e_s=0$, then
\begin{multline*}
q_{i_{s-1}}((x_{i_s}^{e_s}\mathbf v_{s+1}\cdots\mathbf v_mx_{i_m}^{e_m}\mathbf v_{m+1})\varphi)=q_{i_{s-1}}((\mathbf v_{s+1}\cdots\mathbf v_mx_{i_m}^{e_m}\mathbf v_{m+1})\varphi)\\
\geqslant q_{i_s}((\mathbf v_{s+1}\cdots\mathbf v_mx_{i_m}^{e_m}\mathbf v_{m+1})\varphi),
\end{multline*}
since $q_{i_{s-1}}\geqslant q_{i_s}$ and the endomorphism $(\mathbf v_{s+1}\cdots\mathbf v_mx_{i_m}^{e_m}\mathbf v_{m+1})\varphi$ preserves the +-order. If $e_s=1$, then
\begin{multline*}
q_{i_{s-1}}((x_{i_s}^{e_s}\mathbf v_{s+1}\cdots\mathbf v_mx_{i_m}^{e_m}\mathbf v_{m+1})\varphi)=(q_{i_{s-1}}(x_{i_s}\varphi)(\mathbf v_{s+1}\cdots\mathbf v_mx_{i_m}^{e_m}\mathbf v_{m+1})\varphi)\\
=(q_{i_{s}-1}(x_{i_s}\varphi))((\mathbf v_{s+1}\cdots\mathbf v_mx_{i_m}^{e_m}\mathbf v_{m+1})\varphi\geqslant q_{i_s}((\mathbf v_{s+1}\cdots\mathbf v_mx_{i_m}^{e_m}\mathbf v_{m+1})\varphi),
\end{multline*}
since $q_{i_{s-1}}(x_{i_s}\varphi)=q_{i_s-1}(x_{i_s}\varphi)\geqslant q_{i_s}$ and the endomorphism $(\mathbf v_{s+1}\cdots\mathbf v_mx_{i_m}^{e_m}\mathbf v_{m+1})\varphi$ preserves the +-order.

From the inequalities
\[
 q_{i_{s-1}}((\mathbf v_s\cdots \mathbf v_mx_{i_m}^{e_m}\mathbf v_{m+1})\varphi)\geqslant q_{i_s}((\mathbf v_{s+1}\cdots \mathbf v_mx_{i_m}^{e_m}\mathbf v_{m+1})\varphi)
\]
with $s=1,\dots,m$, and the equality $q_{i_{m}}(\mathbf v_{m+1}\varphi)=q_{i_{m}}=q_k$, we conclude that
\[
q(\mathbf v\varphi)=q_0((\mathbf v_1x_{i_1}^{e_1}\mathbf v_2\cdots \mathbf v_mx_{i_m}^{e_m}\mathbf v_{m+1})\varphi)\geqslant q_k=q((x_1\cdots x_k)\varphi).
\]

Since $\mathbf v$ is a summand of the polynomial $\mathbf V$, we have $\mathbf v\preccurlyeq\mathbf V$, and hence $\mathbf v\varphi\leqslant\mathbf V\varphi$. This, together with the inequality $q((x_1\cdots x_k)\varphi)\leqslant q(\mathbf v\varphi)$ established just above, yields  $q((x_1\cdots x_k)\varphi)\leqslant q(\mathbf{V}\varphi)$, as required.
\end{proof}

\begin{corollary}
\label{cor:Cn-finite}
Every incline $\mathcal{S}=(S;\,+,\cdot)$ such that the semilattice $(S;\,+)$ has height $n$ belongs to the variety $\var\mC^-_{n+2}$.
\end{corollary}

\begin{proof}
The proof is dual to that of Corollary~\ref{cor:Cnfinite}.

Take a symbol $\top\notin S$ and extend addition to the set $S^\top:=S\cup\{\top\}$ by letting $\top+t=t+\top:=\top$ for all $t\in S^\top$. Then $(S^\top;\,+)$ becomes a semilattice, which we denote by $\mS^\top$, and  $\top\geqslant t$ in the +-order on $\mS^\top$ for all $t\in S^\top$.

For every $s\in S$, the map $\overline{s}\colon S^\top\to S^\top$ defined by
\[
t\overline{s}:=\begin{cases}
s&\text{if }\ t=\top,\\
ts &\text{if }\ t\in S
\end{cases}
\]
is an endomorphism of the semilattice $\mS^\top$. Indeed, the equality $(t_1+t_2)\overline{s}=t_1\overline{s}+t_2\overline{s}$ holds for all $t_1,t_2\in S$, since multiplication distributes over addition in the incline $\mS$. If $t_1=t_2=\top$, then $t_1+t_2=\top$, and we have $(t_1+t_2)\overline{s}=s$, and $t_1\overline{s}+t_2\overline{s}=s+s=s$. By commutativity of addition, it remains to consider the case where $t_1=t\in S$ and $t_2=\top$. Then $(t_1+\top)\overline{s}=\top s=s$ and $t_1\overline{s}+\top\overline{s}=ts+s$, but $s=ts+s$ in $\mS$ by the absorption law $x\approx yx+x$.

Every endomorphism $\overline{s}$ is decreasing. Indeed, $\top\geqslant\top\overline{s}$, since $\top$ is the greatest element with respect to the +-order on $\mS^\top$, and $t\geqslant ts$ for every $t\in S$, since $\mS$ satisfies \eqref{eq:xy<x}. It is routine to verify that the map $\mS\to\End(\mS^\top)$ defined by $s\mapsto\overline{s}$ is a semiring homomorphism. Since $\overline{s_1}=\overline{s_2}$ implies $s_1=\top\overline{s_1}=\top\overline{s_2}=s_2$ for any $s_1,s_2\in S$, this homomorphism is injective. We conclude that $\mS$ is isomorphic to a subsemiring of $\End(\mS^\top)$ consisting of decreasing endomorphisms.\footnote{The embedding $\mS\hookrightarrow\End(\mS^\top)$ extends the representation of finite inclines from~\cite[Theorem 2.2(2)]{KRM97}.}

The height of the semilattice~$\mS^\top$ is $n+1$, since the longest chains in the partially ordered set $(S^\top;\leqslant)$ are those of $(S;\leqslant)$ with $\top$ appended. By Proposition~\ref{prop:decreasendo}, the semiring $\mS$ satisfies all identities of the Catalan semiring $\mC^-_{n+2}$. Thus, $\mS\in\var\mC_{n+2}$.
\end{proof}

From Corollary~\ref{cor:Cn-finite}, one can deduce Theorem~\ref{thm:pvarCn-} by an obvious adjustment to the above derivation of Theorem~\ref{thm:pvarCn} from Corollary~\ref{cor:Cnfinite}. We therefore omit the proof of Theorem~\ref{thm:pvarCn-}.

\medskip

Denote by $\mathfrak{I}$ the variety of all inclines, that is, the \ais{} variety defined by identities \eqref{eq:xy<x} and \eqref{eq:yx<x}. We are going to identify $\mathfrak{I}$ with the variety generated by the Catalan semirings $\mC^-_n$, $n\ge 2$. We use the same approach as in the proof of Theorem~\ref{thm:varCn}, starting with a model of the free incline.

If a word $\mathbf u$ occurs as a subword in a word $\mathbf v$, we say that $\mathbf v$ is an \emph{expansion} of $\mathbf u$. Recall that $X^+$ denotes the set of all words with variables from $X$. A \emph{shuffle ideal} is a nonempty subset of $X^+$ that is closed under taking expansions. Given a word $\mathbf w$, let $\shuffle(\mathbf w)$ denote the set of all its expansions, which is the least shuffle ideal containing $\mathbf w$. We say that a shuffle ideal is \emph{finitely generated} if it is the union of finitely many ideals of the form~$\shuffle(\mathbf w)$. (If $X$ is a finite set, then Higman's Lemma~\cite{Higman52} implies that every shuffle ideal of $X^+$ is finitely generated; however, this is not the case if $X$ is infinite.) Every finitely generated shuffle ideal $W$ has a unique \emph{basis} $B(W)$---a finite antichain under the subword order---such that $W=\bigcup_{\mathbf w\in B(W)}\shuffle(\mathbf w)$, that is, every word from $W$ is an expansion of a word from $B(W)$.

Denote by $SI_f(X)$ the set of all finitely generated shuffle ideals of $X^+$ and endow it with addition + and multiplication $\cdot$ by setting, for all $U,V\in SI(X)$,
\[
U+V:=U\cup V, \quad U\cdot V:=\{\mathbf{uv}\mid \mathbf u\in U,\ \mathbf v\in V\}.
\]
It is easy to see that $U+V,U\cdot V\in SI_f(X)$ for all $U,V\in SI_f(X)$ and that $\mathcal{SI}_f(X):=(SI_f(X);\,+,\cdot)$ is an \ais. The following observation is likely known, but since we were not able to locate it in the literature, we have included its proof.

\begin{lemma}
\label{lem:freeincline}
The semiring $\mathcal{SI}_f(X)$ is the free incline on $X$.\footnote{The model for the free incline suggested in~\cite{WLHY09} is less suited for our purposes, even though any two models of a free object on the same set are, of course, isomorphic.}
\end{lemma}

\begin{proof}
First, note that $\mathcal{SI}_f(X)$ is an incline. Indeed, since $\mathbf{uv}$ is an expansion of both $\mathbf{u}$ and $\mathbf{v}$, we have $U\supseteq U\cdot V$ and $U\supseteq V\cdot U$ for all $U,V\in SI_f(X)$, and hence the identities $x\succcurlyeq xy$ and $x\succcurlyeq yx$, which define the variety of all inclines, hold in $\mathcal{SI}_f(X)$.

The map $\iota\colon X\to SI_f(X)$ defined by $x\iota:=\shuffle(x)$ for all $x\in X$ is injective. It is easy to see that if $\mathbf w=x_1\cdots x_k$, then $\shuffle(\mathbf w)=\shuffle(x_1)\cdots\shuffle(x_k)$. Hence, the set $X\iota$ generates the semiring $\mathcal{SI}_f(X)$, and so condition (F1) in the definition of a free algebra holds.

Let $(S;\,+,\cdot)$ be an arbitrary incline and $\mu\colon X\to S$ an arbitrary substitution. We extend $\mu$ (keeping the same notation) to a map $\mu\colon X^+\to S$ by setting $\mathbf w\mu:=x_1\mu\cdots x_k\mu$ for every word $\mathbf w=x_1\cdots x_k$. Now, let $W\varphi:=\sum_{{\mathbf w\in B(W)}}\mathbf w\mu$ for all $W\in SI_f(X)$; recall that the basis $B(W)$ is finite, so the sum is well-defined. We have  $x\mu=(\shuffle(x))\varphi=(x\iota)\varphi$ for all $x\in X$ since $B(\shuffle(x))=x$.

To complete the verification of condition (F2) in the definition of a free algebra, it remains to show that $\varphi$ is a homomorphism. The uniqueness of $\varphi$ then follows from the fact that its value on each generator $x\iota$ is uniquely determined to be $x\mu$.

By construction, for all $U,V\in SI_f(X)$, we have $(U+V)\varphi=\sum_{\mathbf w\in B(U\cup V)}\mathbf w\mu$, while $U\varphi+V\varphi=\sum_{\mathbf u\in B(U)}\mathbf u\mu+\sum_{\mathbf v\in B(V)}\mathbf v\mu$. Every word $\mathbf w\in B(U\cup V)$ lies in one of the sets $U$ or $V$. Since $\mathbf w$ is minimal under the subword order in $U\cup V$, it is also minimal under the subword order in any subset of $U\cup V$ to which $\mathbf w$  belongs. Hence, $\mathbf w\in B(U)\cup B(V)$, and so every summand of $\sum_{\mathbf w\in B(U\cup V)}\mathbf w\mu$ appears in $\sum_{\mathbf u\in B(U)}\mathbf u\mu+\sum_{\mathbf v\in B(V)}\mathbf v\mu$.

Every word $\mathbf u\in B(U)$ belongs to $U\cup V$, and therefore, $\mathbf u$ is an expansion of some word $\mathbf w\in B(U)\cup B(V)$. By Corollary~\ref{cor:subword}, the identity $\mathbf{u}\preccurlyeq\mathbf{w}$ holds in every incline, and hence $\mathbf{u}\mu+\mathbf{w}\mu=\mathbf{w}\mu$ in the semiring $\mS$. Hence, $\sum_{\mathbf u\in B(U)}\mathbf u\mu+\sum_{\mathbf w\in B(U\cup V)}\mathbf w\mu=\sum_{\mathbf w\in B(U\cup V)}\mathbf w\mu$, and similarly, $\sum_{\mathbf v\in B(V)}\mathbf v\mu+\sum_{\mathbf w\in B(U\cup V)}\mathbf w\mu=\sum_{\mathbf w\in B(U\cup V)}\mathbf w\mu$. We thus conclude that
\[
\sum_{\mathbf u\in B(U)}\mathbf u\mu+\sum_{\mathbf v\in B(V)}\mathbf v\mu=\sum_{\mathbf u\in B(U)}\mathbf u\mu+\sum_{\mathbf v\in B(V)}\mathbf v\mu+\sum_{\mathbf w\in B(U\cup V)}\mathbf w\mu=\sum_{\mathbf w\in B(U\cup V)}\mathbf w\mu,
\]
that is, $U\varphi+V\varphi=(U+V)\varphi$.

Similarly, for all $U,V\in SI_f(X)$, we have $(U\cdot V)\varphi=\sum_{\mathbf w\in B(U\cdot V)}\mathbf w\mu$, whereas
\[
U\varphi\cdot V\varphi=\left(\sum_{\mathbf u\in B(U)}\mathbf u\mu\right)\cdot\left(\sum_{\mathbf v\in B(V)}\mathbf v\mu\right)=\sum_{\substack{\mathbf u\in B(U),\\ \mathbf v\in B(V)}}(\mathbf{uv})\mu.
\]
Every word $\mathbf w\in B(U\cdot V)$ can be decomposed as $\mathbf w=\mathbf u'\mathbf v'$ for some words $\mathbf u'\in U$ and $\mathbf v'\in V$. The words $\mathbf u'$ and $\mathbf v'$ are expansions of some words $\mathbf u\in B(U)$ and $\mathbf v\in B(V)$, respectively, and hence $\mathbf w$ is an expansion of the word $\mathbf{uv}\in U\cdot V$. Since the word $\mathbf w$ is minimal under the subword order in $U\cdot V$, this implies $\mathbf w=\mathbf{uv}$. Thus, every summand of $\sum_{\mathbf w\in B(U\cdot V)}\mathbf w\mu$ appears in $\sum_{\mathbf u\in B(U),\ \mathbf v\in B(V)}(\mathbf{uv})\mu$.

For every $\mathbf u\in B(U)$ and $\mathbf v\in B(V)$, the product $\mathbf{uv}\in U\cdot V$ is an expansion of a word $\mathbf w\in B(U\cdot V)$. As above, Corollary~\ref{cor:subword} implies that the identity $\mathbf{uv}\preccurlyeq\mathbf{w}$ holds in every incline, and hence $(\mathbf{uv})\mu+\mathbf{w}\mu=\mathbf{w}\mu$ in $\mS$. Therefore,
\[
\sum_{\substack{\mathbf u\in B(U),\\ \mathbf v\in B(V)}}(\mathbf{uv})\mu=\sum_{\substack{\mathbf u\in B(U),\\ \mathbf v\in B(V)}}(\mathbf{uv})\mu+\sum_{\mathbf w\in B(U\cdot V)}\mathbf w\mu=\sum_{\mathbf w\in B(U\cdot V)}\mathbf w\mu,
\]
that is, $U\varphi\cdot V\varphi=(U\cdot V)\varphi$.
\end{proof}

We now derive the equality $\var\{\mC^-_n\}_{n\ge 2}=\mathfrak{I}$ from Lemma~\ref{lem:freeincline} and Theorem~\ref{thm:pvarCn-} by an argument parallel to that used to deduce Theorem~\ref{thm:varCn} from Lemma~\ref{lem:freeCn} and Theorem~\ref{thm:pvarCn}.

\begin{theorem}
\label{thm:varCn-}
The variety of all inclines is generated by the semirings $\mC^-_n$ for $n=2,3,\dotsc$.
\end{theorem}

\begin{proof}
The inclusion $\var\{\mC^-_n\}_{n\ge 2}\subseteq\mathfrak{I}$ is clear, since every Catalan semiring $\mC^-_n$ is an incline. As every finite incline belongs to $\var\{\mC^-_n\}_{n\ge 2}$, the reverse inclusion will follow once we show that the free incline on an infinite set is residually finite. By Lemma~\ref{lem:freeincline}, it suffices to verify that the incline $\mathcal{SI}_f(X)$, with $X$ infinite, is residually finite.

Take two arbitrary distinct elements $W_1,W_2\in SC_f(X)$. Since $W_1$ and $W_2$ are finitely generated shuffle ideals, their bases $B(W_1)$ and $B(W_2)$ are finite and distinct. There exist a finite subset $Y\subset X$ and an integer $k$ such that $\alf(\mathbf w)\subseteq Y$ and $|\mathbf{w}|<k$ for all words $\mathbf w\in B(W_1)\cup B(W_2)$. Now we define a congruence on the semiring $\mathcal{SI}_f(X)$, analogous to that used in the proof of Theorem~\ref{thm:varCn}: for all $U,V\in SI_f(X)$,
\[
U\mathrel{\rho}V \iff U\cap Y^{<k}=V\cap Y^{<k}.
\]
The quotient semiring $\mathcal{SI}_f(X)/\rho$ is finite. With $Y$ and $k$ as chosen, $B(W_1),B(W_2)\subseteq Y^{<k}$. Clearly, $B(W_1)$ and $B(W_2)$ consist of all words minimal under subword order in $W_1\cap Y^{<k}$ and $W_2\cap Y^{<k}$, respectively. Therefore,  $B(W_1)\ne B(W_2)$ implies $W_1\cap Y^{<k}\ne W_2\cap Y^{<k}$, so $W_1$ and $W_2$ have distinct images under the natural homomorphism $\mathcal{SI}_f(X)\to\mathcal{SI}_f(X)/\rho$.
\end{proof}

\subsection{\texorpdfstring{Variety generated by the semirings $\mC_n$ and $\mC^-_n$, $n\ge2$}{Variety generated by the semirings Cn and Cn}}
\label{subsec:varCnCn-}
We have determined the variety generated by the semirings $\mC_n$, $n\ge2$, as well as the variety generated by the semirings $\mC^-_n$, $n\ge2$. It is natural to ask which variety is generated by all Catalan semirings. In view of Theorems~\ref{thm:varCn} and~\ref{thm:varCn-}, the variety in question is the join of the varieties $\mathfrak{C}$ and $\mathfrak{I}$ in the lattice of \ais{} varieties. The only thing we can say about this join at the moment is that it is \nfb. 

\begin{proposition}
\label{prop:allCatalan}
The variety generated by all semirings $\mC_n$ and $\mC^-_n$ with $n=2,3,\dots$ admits no finite identity basis.
\end{proposition}

\begin{proof}
  The proof of Theorem~\ref{thm:nfb-CkxCk-} applies to any variety containing the semirings $\mC_3$ and $\mC^-_3$ and contained in the variety $\mathfrak{T}$. In particular, it applies to the variety generated by all Catalan semirings $\mC_n$ and $\mC^-_n$ with $n=2,3,\dotsc$
\end{proof}

\section{Discussing related results and open questions}
\label{sec:conclusion}

\subsection{Identities of Catalan semirings vs. identities of Catalan monoids}
\label{subsec:comparison}
It is worth comparing the equational properties of the Catalan semirings studied in the present paper with those of their multiplicative reducts, that is, the Catalan monoids. First, we recall the concept of strong nonfinite basability.

Let $\mathfrak{A}$ be a class of algebras. A finite algebra $\mA\in\mathfrak{A}$ is said to be \emph{strongly \nfb} in $\mathfrak{A}$ if $\mA$ is not contained in any \fb{} variety generated by a finite algebra from  $\mathfrak{A}$. Thus, for finite algebras, strong nonfinite basability is the most contagious form of nonfinite basability: if $\mA$ is strongly \nfb, so is every finite algebra $\mB\in\mathfrak{A}$ such that $\mA\in\var\mB$.

For $n\ge 5$, the Catalan monoid $C_n$ is known to be \nfb~\cite{Vo04} and has recently been shown to be strongly nonfinitely based in the class of all semigroups as follows from combining~\cite[Theorem~8]{GSV25} with~\cite[Proposition~2.7]{SaVo23}. In contrast, the Catalan semirings $\mC_n$ and $\mC^-_n$ are \fb{} by Theorems~\ref{thm:basisCn} and~\ref{thm:Cn-}. It is hardly a novelty that the finite basis property is sensitive to extending the set of operations as such an extension both enlarges the set of possible identities and strengthens inference power. However, so far all examples where the equational properties of an \ais{} $(S;\,+,\cdot)$ differ from those of its multiplicative semigroup $(S;\,\cdot)$ have been one-directional: the semigroup was \fb, while the \ais{} was not. The 3-element \nfb{} semiring $\mS_7$ mentioned in the introduction is such an example; the well-known \ib{} of its multiplicative semigroup consists of the identities $x^2\approx x^3$ and $xy\approx yx$. There are even examples of \ais{}s that are strongly \nfb{} in the class of all \ais{}s and have \fb{} multiplicative reducts \cite[Theorem 6.1]{JackRenZhao:2022}. The Catalan semirings $\mC_n$ and $\mC^-_n$ with $n\ge5$ yield, to the best of our knowledge, the first examples where the difference goes in the opposite direction.

Interestingly, if we pass from individual Catalan monoids and semirings to their collectives, the direction reverts to the `standard' one. It is easy to show that the only identities common to all Catalan monoids are consequences of the associative law, whereas the identities common to all semirings $\mC_n$ and $\mC^-_n$ ($n\ge2$) have no finite basis by Proposition~\ref{prop:allCatalan}.

\subsection{\emph{i}-Catalan semirings} A \emph{\pat} of the set $[n]=\{1,2,\dots,n\}$ is a map $\alpha$ from a subset of $[n]$ (called the \emph{domain} of $\alpha$ and denoted $\dom\alpha$) to $[n]$. A \pat{} $\alpha$ is \emph{order-preserving} if $i\le j$ implies $i\alpha\le j\alpha$ for all $i,j\in\dom\alpha$, and \emph{extensive} if $i\le i\alpha$ for every $i\in\dom\alpha$.  The set of all extensive and order-preserving partial injections of $[n]$ (including the nowhere defined map) forms a submonoid of the monoid of all partial transformations of $[n]$. Following~\cite{SaVo23}, we denote this submonoid by $IC_n$ and call it the $i$-\emph{Catalan monoid}. Both `I' and `$i$' in the name stand for `injective', and the `Catalan' part of the name again refers to the cardinality of the monoid: $|IC_n|$ is the $(n+1)$-th Catalan number; see \cite[Theorem~14.2.8(ii)]{GaMa09} for a proof based on a well-known recurrence for Catalan numbers, or \cite[Appendix A]{SaVo23} for a bijective proof.

Like the `standard' Catalan monoids, the $i$-Catalan monoids are related to certain classes of recognizable languages. Their identities were studied in~\cite{Go06a,HCL15,SaVo23,GSV25}, where it was established that the monoid $IC_n$ is \nfb{} if and only if $n\ge3$ and is strongly \nfb{} in the class of all semigroups if and only if $n\ge4$. It is also known that $\var IC_n\subseteq\var C_{n+1}$ and the inclusion is strict for $n\ge2$; see \cite[Proposition 2.7]{SaVo23}.

The monoid $I_n$ of all partial injections on $[n]$ (known as the \emph{symmetric inverse monoid}) admits a natural addition defined as follows: for $\alpha,\beta\in I_n$, their sum is their greatest common restriction, that is,
\[
\dom(\alpha+\beta):=\{i\in\dom\alpha\cap\dom\beta\mid i\alpha=i\beta\},
\]
and $i(\alpha+\beta):=i\alpha$ for each $i\in\dom(\alpha+\beta)$. Equipped with this addition, $I_n$ becomes an \ais{} in which $IC_n$ forms a subsemiring. We denote this subsemiring by $\mathcal{IC}_n$ and call it the $i$-\emph{Catalan semiring}.

To complement our study of Catalan semirings, we report the status of the Finite Basis Problem for $i$-Catalan semirings.
\begin{proposition}
\label{prop:basis-iCn}
For any $n\ge2$, the semiring $\mathcal{IC}_n$ is \nfb.
\end{proposition}

The semiring $\mathcal{IC}_n$ for $n\ge 2$ falls within the scope of a powerful sufficient condition for the absence of a finite \ib{} established by Jackson, Ren, and Zhao~\cite{JackRenZhao:2022}. We state this condition in a form convenient for applying it to $\mathcal{IC}_n$. The statement uses two notions that we now recall. An element $g$ of an \ais{} is called \emph{cyclic} if $g=g^k$ for some $k>1$. A subset $U$ of a partially ordered set $(P;\le)$ is called an \emph{upper set} if $U$ contains all elements $b\in P$ such that $a\le b$ for some $a\in U$. The following proposition can be obtained by combining Theorem~4.9 and Proposition~2.6 in~\cite{JackRenZhao:2022}.

\begin{proposition}
\label{prop:cyclicfilter}
A finite \ais{} $\mS=(S;\,+,\cdot)$ whose cyclic elements form an upper set with respect to the $+$-order is \nfb{} whenever the variety $\var\mS$ contains the semiring $\mS_7$.
\end{proposition}

\begin{proof}[Proof of Proposition~\ref{prop:basis-iCn}] For $\alpha,\beta\in\mathcal{IC}_n$, if $\alpha\leqslant\beta$, then $\alpha+\beta=\beta$, which means that $\dom\beta\subseteq\dom\alpha$ and $\beta$ is the restriction of $\alpha$ to $\dom\beta$. Clearly, if a \pat{} is a cyclic element, then so is its restriction to any subset of its domain. Thus, the cyclic elements of $\mathcal{IC}_n$ form an upper set with respect to the $+$-order.

Recall that $\mS_7=\left(\{\mathit{0},a,\mathit{1}\};\,+,\cdot\right)$, with addition and multiplication defined by~\eqref{eq:s7}. It is easy to verify that, for any $n\ge2$, the map $\mS_7\to\mathcal{IC}_n$ that sends $\mathit{0}$ to the nowhere defined transformation, $a$ to the transformation with domain $\{1\}$ that maps 1 to 2, and $\mathit{1}$ to the identity transformation, is a semiring homomorphism.  Therefore, $\mS_7$ is isomorphic to a subsemiring of $\mathcal{IC}_n$, so the variety $\var\mathcal{IC}_n$ contains $\mS_7$.

We conclude that Proposition~\ref{prop:cyclicfilter} applies to $\mathcal{IC}_n$.
\end{proof}

\begin{remark}
Recent preprints~\cite{WYW26a} and~\cite{ys26a} suggest that the semiring $\mS_7$ is strongly \nfb. If this is indeed the case, then each semiring $\mathcal{IC}_n$ with $n\ge 2$ is also strongly \nfb.
\end{remark}

We note that, unlike the semigroup varieties generated by the monoids $C_{n+1}$ and $IC_n$, the semiring varieties $\var\mC_{n+1}$ and $\var\mathcal{IC}_n$ are incomparable for $n\ge 2$, as are the varieties $\var\mC^{-}_{n+1}$ and $\var\mathcal{IC}_n$. Indeed, the semiring $\mS_7$ embeds in $\mathcal{IC}_n$ (as shown above), yet satisfies neither \eqref{eq:xy>x} (which holds in $\mC_{n+1}$) nor \eqref{eq:xy<x} (which holds in $\mC^-_{n+1}$).

\subsection{Open problems}
We discuss two groups of problems related to Catalan semirings.

1. While every Catalan semiring $\mC_n$ is finitely based (Theorem~\ref{thm:basisCn}), the variety $\var\mC_3$ (and hence each variety $\var\mC_n$ with $n\ge3$) contains many \nfb{} \ais{}s. We have already mentioned that the semiring $\mC_3$ has a subsemiring isomorphic to the 3-element semiring $S_{53}$; see the discussion following the proof of Theorem~\ref{thm:nfb-CkxCk-}. Gao~\cite[Theorem~4.5]{Gao27} has shown that $S_{53}$ generates a variety with continuum many subvarieties. Since there are only countably many \fb{} \ais{} varieties, continuum many subvarieties of $\var S_{53}$ are \nfb. Moreover, as the proof of~\cite[Theorem~4.5]{Gao27} reveals, $\var S_{53}$ (and hence $\var\mC_n$ for all $n\ge3$) contains finite \nfb{} \ais{}s.

In contrast, to the best of our knowledge, it is still unknown how many subvarieties the variety $\var\mC^-_n$ with $n\ge3$ contains. The variety $\mathfrak{I}$ of all inclines has infinitely many subvarieties since all varieties $\var\mC^-_n$, $n=2,3,\dotsc$, are distinct, but we do not know if $\mathfrak{I}$ has countably or uncountably many subvarieties. Nor are we aware of any example of a finite \nfb{} incline.

2. We have shown that both the pseudovariety generated by all semirings $\mC_n$, $n=2,3,\dotsc$, and that generated by all semirings $\mC^-_n$, $n=2,3,\dotsc$, are equational, that is, each of them is the class of all finite semirings in some variety; see Theorems~\ref{thm:pvarCn} and~\ref{thm:pvarCn-}, respectively. However, we do not know whether the pseudovariety generated by all Catalan semirings $\mC_n$ and $\mC^-_n$, $n=2,3,\dotsc$, is equational. 

Let $\mathfrak{C}^{\pm}$ denote  the variety generated by all Catalan semirings $\mC_n$ and $\mC^-_n$, $n=2,3,\dotsc$. By Proposition~\ref{prop:allCatalan}, the variety $\mathfrak{C}^{\pm}$  has no finite identity basis. It would be interesting to have an explicit infinite basis for $\mathfrak{C}^{\pm}$. Lemma~\ref{lem:subword} and Corollary~\ref{cor:subword} imply that $\mathfrak{C}^{\pm}$ satisfies all identities of the form $\mathbf w\preccurlyeq\mathbf u + \mathbf v$, where $\mathbf u$ is a subword of $\mathbf w$ and $\mathbf w$ is a subword of $\mathbf v$. Do all such identities constitute a basis for $\mathfrak{C}^{\pm}$?

\end{document}